\documentclass{article}

\usepackage{arxiv}

\usepackage{amsfonts}
\usepackage{mathrsfs}
\usepackage{amsmath}
\usepackage{amsthm}
\usepackage{amssymb}
\usepackage{setspace}
\usepackage{bbm}
\usepackage{hyperref}
\usepackage{mathtools}
\usepackage[color,matrix,arrow]{xy}
\usepackage[numbers,sort&compress]{natbib}

\usepackage[utf8]{inputenc} 
\usepackage[T1]{fontenc}    
\usepackage{hyperref}       
\usepackage{url}            
\usepackage{booktabs}       
\usepackage{amsfonts}       
\usepackage{nicefrac}       
\usepackage{microtype}      
\usepackage{lipsum}		

\theoremstyle{plain}
\newtheorem{theorem}{Theorem}[section]
\newtheorem{lemma}[theorem]{Lemma}
\newtheorem{proposition}[theorem]{Proposition}
\newtheorem{corollary}[theorem]{Corollary}
\theoremstyle{definition}
\newtheorem{definition}{Definition}[section]
\newtheorem{problem}{Problem}[section]
\newtheorem{example}{Example}[section]

\newcommand{\R}{\mathbb{R}}

\newcommand{\Q}{\mathbb{Q}}

\numberwithin{equation}{section}

\theoremstyle{remark}
\newtheorem*{remark}{Remark}

\title{On order preserving and order reversing mappings defined on cones of convex functions}

\date{} 					

\author{
  Lixin Cheng\thanks{Corresponding author} \\
  School of Mathematical Sciences\\
  Xiamen University\\
  Xiamen {\rm 361005}, China\\
  \texttt{lxcheng@xmu.edu.cn} \\
   \And
 Sijie Luo \\
 Yau Mathematical Sciences Center\\
 Tsinghua University\\
 Beijing {\rm 100084}, China\\
  \texttt{luosijie@mail.tsinghua.edu.cn} \\
}

\begin{document}
\maketitle

\begin{abstract}
In this paper, we first show that for a Banach space $X$ there is a fully order reversing mapping $T$ from ${\rm conv}(X)$ (the cone of all extended real-valued lower semicontinuous proper convex functions defined on $X$) onto itself if and only if $X$ is reflexive and linearly isomorphic to its dual $X^*$.  Then we further prove the following generalized ``Artstein-Avidan-Milman'' representation theorem: For every fully order reversing mapping $T:{\rm conv}(X)\rightarrow {\rm conv}(X)$ there exist a linear isomorphism $U:X\rightarrow X^*$, $x_0^*, \;\varphi_0\in X^*$, $\alpha>0$ and $r_0\in\mathbb R$  so that
\begin{equation}\nonumber
(Tf)(x)=\alpha(\mathcal Ff)(Ux+x^*_0)+\langle\varphi_0,x\rangle+r_0,\;\;\forall x\in X,
\end{equation}
where  $\mathcal F: {\rm conv}(X)\rightarrow {\rm conv}(X^*)$ is the Fenchel transform. Hence, these resolve two open questions.
We also show several representation theorems of fully order preserving mappings defined on certain cones of convex functions. For example,
for every fully order preserving mapping $S:{\rm semn}(X)\rightarrow {\rm semn}(X)$ there is a linear isomorphism $U:X\rightarrow X$ so that
\begin{equation}\nonumber
(Sf)(x)=f(Ux),\;\;\forall f\in{\rm semn}(X),\;x\in X,
\end{equation}
where ${\rm semn}(X)$ is the cone of all lower semicontinuous seminorms on $X$.
\end{abstract}


\section{Introduction}

An elegant theorem of Artstein-Avidan and Milman \cite{A-A-M09} states that every fully order reversing (resp. order preserving) mapping $T$ of the cone ${\rm conv}(\mathbb R^n)$ of all extended real-valued proper convex functions defined on $\mathbb R^n$ is essentially the Legendre transform (resp. the identity). More precisely,
\begin{theorem}[Artstein-Avidan and Milman]
\indent
\begin{itemize}
\item[i)] Every fully  order reversing mapping $S: {\rm conv}(\mathbb R^n)\rightarrow {\rm conv}(\mathbb R^n)$ has the following form:
\begin{equation}
(Sf)(x)=\alpha(\mathcal Lf)(Ex+u_0)+\langle v_0,x\rangle+r_0,\;\;f\in{\rm conv}(\mathbb R^n),\;\;x\in\mathbb R^n;
\end{equation}
\item[ii)] Every fully  order preserving mapping $T: {\rm conv}(\mathbb R^n)\rightarrow {\rm conv}(\mathbb R^n)$ has the following form:
\begin{equation}
(Tf)(x)=\alpha f(Ex+u_0)+\langle v_0,x\rangle+r_0,\;\;f\in{\rm conv}(\mathbb R^n),\;\;x\in\mathbb R^n,
\end{equation}
for some isomorphism $E: \mathbb R^n\rightarrow\mathbb R^n$, $u_0$, $v_0\in\mathbb R^n$, $\alpha>0$ and $r_0\in \mathbb{R}$, where
\[
\mathcal L(f):{\rm conv}(\mathbb R^n)\rightarrow{\rm conv}(\mathbb R^n)
\]
is the Legendre transform defined for $f\in{\rm conv}(\mathbb R^n)$ by
\begin{equation}
(\mathcal Lf)(x)=\sup\{\langle x,v\rangle-f(v):\;v\in\mathbb R^n\},\;\;x\in\mathbb R^n.
\end{equation}
\end{itemize}
\end{theorem}

A further question we concern most is about the behavior of the ``Artstein-Avidan-Milman'' theorem if we substitute a general Banach space $X$ for $\mathbb R^n$. The question is apparently natural and  worth  considering, which can be divided into the following three more concrete questions.

\begin{problem}\label{reversing existence}
For what infinite dimensional Banach spaces $X$, does there
exist a fully order reversing mapping $T: {\rm conv}(X)\rightarrow {\rm conv}(X)$?
\end{problem}

\begin{problem}\label{A-A-M theorem}
Does the ``Artstein-Avidan-Milman'' theorem hold again for all fully order reversing mappings $T: {\rm conv}(X)\rightarrow {\rm conv}(X)$?
\end{problem}

\begin{problem}\label{A-A-M theorem1}
Does the ``Artstein-Avidan-Milman'' theorem hold true for every order preserving mapping $S: {\rm conv}(X)\rightarrow {\rm conv}(X)$?
\end{problem}

Recall that a convex function $f$ defined on a Banach space $X$ is said to be proper if it is nowhere $-\infty$-valued and with its essential domain ${\rm dom}(f)\neq\emptyset$. For a partially ordered set $P$,  a mapping $T:P\rightarrow P$ is said to be  fully order preserving (resp. reversing) provided it is a bijection and satisfies $f\geq g\Longleftrightarrow Tf\geq Tg$ ( resp. $f\geq g\Longleftrightarrow Tf\leq Tg$).  Note that in a Banach space $X$ the Legendre transform $\mathcal L: {\rm conv}(\mathbb R^n)\rightarrow{\rm conv}(\mathbb R^n)$ becomes into the following Fenchel transform:
\[
\mathcal F: {\rm conv}(X)\rightarrow{\rm conv}_{*}(X^*),
\]
which is defined for $f\in {\rm conv}(X)$ by
\begin{equation}\nonumber
(\mathcal F f)(x^*)=\sup\{\langle x^*,x\rangle-f(x):x\in X\},\;\;x^*\in X^*,
\end{equation}
where ${\rm conv}(X)$ (resp. ${\rm conv}_{*}(X^*)$) is the cone of all extended real-valued lower semicontinuous (resp. $w^*$-lower semicontinuous) proper convex functions on $X$ (resp. $X^*$).

In 2015, Iusem, Reem and Svaiter \cite{I-R-S}   generalized the ``Artstein-Avidan-Milman'' theorem  to a general Banach space $X$ in the following manner.

\begin{theorem}[Iusem-Reem-Svaiter]\label{Iusem-Reem-Svaiter}
Suppose that $X$ is a Banach space. Then,
\begin{itemize}
\item[i)] for every fully  order reversing mapping $S: {\rm conv}(X)\rightarrow {\rm conv}_{*}(X^*)$ there exists an isomorphism $U: X\rightarrow X$, $x_0\in X$, $x^*_0\in X^*$, $\alpha>0$, and $r_0\in\mathbb R$ so that
\begin{equation}\nonumber
(Sf)(x^*)=\alpha(\mathcal Ff)(U^*x^*+x^*_0)+\langle x^*,x_0\rangle+r_0,\;\;f\in{\rm conv}(X),\;\;x^*\in X^*;
\end{equation}
\item[ii)] for every fully  order preserving mapping $T: {\rm conv}(X)\rightarrow {\rm conv}(X)$ there exists an isomorphism $U: X\rightarrow X$, $x_0\in X$, $x^*_0\in X^*$, $\alpha>0$, and $r_0\in\mathbb R$ so that
\begin{equation}\nonumber
(Tf)(x)=\alpha f(Ux+x_0)+\langle x^*_0,x\rangle+r_0,\;\;f\in{\rm conv}(X),\;\;x\in X.
\end{equation}
\end{itemize}
\end{theorem}

The second conclusion of the ``Iusem-Reem-Svaiter'' theorem \big{(}Theorem \ref{Iusem-Reem-Svaiter} {\rm ii)} \big{)} can be regarded as a positive answer to Problem \ref{A-A-M theorem1}, i.e. a perfect extension of  the ``Artstein-Avidan-Milman '' theorem for fully order preserving mappings. But Problems \ref{reversing existence} and \ref{A-A-M theorem} remain open.
The first conclusion \big{(}Theorem \ref{Iusem-Reem-Svaiter} {\rm i)}\big{)} can be understood as a characterization of fully  order reversing mappings $T: {\rm conv}(X)\rightarrow {\rm conv}(X^*)$ whenever $X$ is reflexive. ``However, the issue of characterizing fully order reversing mappings from ${\rm conv}(X)$ to itself (or, from ${\rm conv}(X)$ to ${\rm conv}(X^*)$ in nonreflexive case) has left as an open problem, deserving future research \cite[p.89]{I-R-S}.'' Iusem, Reem and Svaiter \cite{I-R-S} further pointed out: In fact,  we don't even know whether there is a fully  order reversing mapping $T: {\rm conv}(X)\rightarrow {\rm conv}(X)$ (or, ${\rm conv}(X)\rightarrow {\rm conv}(X^*)$ in the nonreflexive case).

Order preserving isomorphisms defined on sets with specific structures had been studied for some time, in some cases in connection with applications to physics, however, Artstein-Avidan and Milman's work \cite{A-A-M09} was the first to deal with
this issue in connection with the cone of convex functions and results in \cite{A-A-M09} were the starting point of several interesting developments. For further information in this direction we refer to \cite{{A-A-F-M12},{A-A-M07},{A-A-M08},{A-A-M11},{I-R-S}} and references therein.

The main results of this paper are as follows. (Their proofs are starting from Section 7.) The next result ({Theorem \ref{reflexive selfdual}}) presents a characterization of a Banach space $X$ for which there exists  a fully order reversing mapping on ${\rm conv}(X)$. Therefore,  it gives {Problem \ref{reversing existence}} a complete answer.

\begin{theorem}\label{reflexive selfdual}
For  a Banach space $X$ there is a fully order reversing mapping
$S: {\rm conv}(X)\rightarrow {\rm conv}(X)$  if and only if $X$ is reflexive and linearly isomorphic to its dual $X^*$.
\end{theorem}
The next theorem gives Problem 1.3 an affirmative answer, and it can be understood as a perfect extension of the ``Artstein-Avidan-Milman'' theorem for order reversing mappings.
\begin{theorem}\label{A-A-M theorem existence}
Suppose that $X$ is a Banach space. Then for every fully order reversing mapping $S:{\rm conv}(X)\rightarrow {\rm conv}(X)$
there exists a linear isomorphism $U:X\rightarrow X^*$, $x_0^*, \;\varphi_0\in X^*$, $\alpha>0$ and $r_0\in\mathbb R$  so that
\begin{equation}\nonumber
(Sf)(x)=\alpha(\mathcal Ff)(Ux+x^*_0)+\langle\varphi_0,x\rangle+r_0,\;\;\forall x\in X,
\end{equation}
where  $\mathcal F: {\rm conv}(X)\rightarrow {\rm conv}(X^*)$ is the Fenchel transform.
\end{theorem}

Let ${\rm subl}(X)$, ${\rm mink}(X)$ and ${\rm semn}(X)$ be the three common cones of extended-real-valued lower semi continuous convex functions defined on the Banach space $X$, successively: all sublinear functions, all Minkowski functionals and all seminorms.

\begin{theorem}\label{sublinear functions}
Let $T:{\rm subl}(X)\to {\rm subl}(X)$ be a fully order preserving mapping, then there exists a linear isomorphism $E:\to X$ and a bounded linear functional $x^{*}_{0}\in X^{*}$ so that
\[
(Tf)(x)=f(Ex)+\langle x^{*}_{0},x\rangle \;\;\forall x\in X,~f\in{\rm subl}(X).
\]
\end{theorem}

\begin{theorem}\label{some classes representation}
Suppose that $C$ is one of the two cones $\{{\rm mink}(X), {\rm semn}(X)\}$. Then for every fully order preserving mapping $T: C\rightarrow C$ there is a linear isomorphism
$E:X\rightarrow X$ so that
\[
(Tf)(x)=f(Ex), \;\;\forall x\in X,~ f\in C.
\]
\end{theorem}

We should also mention that our proofs are different from those of previous known results:

1. Through introducing the concepts of bounded sup-complete cone and sup-generating class of cones consisting of convex functions, and discussing their properties, we show that  every fully  order preserving mapping defined on a bounded sup-complete cone  of convex functions is bounded-continuous with respect to the topology induced by the pointwise convergence; and is affine, whenever it is restricted to its perfect sup-generating class. These facts  make the proofs of { Theorems \ref{reflexive selfdual}} and { \ref{A-A-M theorem existence}} much simpler.

2. Though fully order preserving mappings $T: C\rightarrow C$, when $C$ is ${\rm subl}(X)$, ${\rm mink}(X)$ or ${\rm semn}(X)$
(Theorem \ref{sublinear functions} and Theorem \ref{some classes representation}), share quite similar form of representations, however, their proofs are entirely different. In the case that $C={\rm subl}(X)$ we convert sublinear functions to the set-operation of the images of their subdifferentials. To show the theorem for $C={\rm mink}(X)$, we first extend the fundamental theorem of (finite dimensional) affine (or, projective) geometry to infinite dimensional spaces. Then, making use of this generalization,  we show that every fully order preserving self-mapping defined on ${\rm mink}(X)$ is again a fully order preserving self-mapping restricted to ${\rm semn}(X)$.

3. We identify ${\rm aff}(X)$ (the space of all continuous affine functionals on $X$) with $X^*\oplus\mathbb R$, and use the ``new'' property that every convex function defined on a Banach space $X$ is a sublinear function on $X\oplus\mathbb R$ restricted to the hyperplane $X\oplus\{1\}$ ({ Proposition \ref{sublinearconvex}}), this fact is used to prove {Theorems \ref{presrving}}, {\ref{JFA}} and {\ref{selfdual}}.

Our paper is organized as follows. The next section ({Section 2}) collects some preliminary facts from convex analysis which will be used throughout this paper. In {Section 3}, some notions such as sup-generating class, purity, atomic of sup-generating class will be introduced and discussed. In {Section 4}, some basic properties such as continuity of fully order preserving mappings defined on cones of convex functions will be discussed.  We will see in {Section 5} that fully order preserving mappings behave nicely when restricted to their certain sup-generating classes. In {Section 6}, we will return to discuss fully order preserving mappings defined on the cones ${\rm conv}(X)$ and ${\rm subl}(X)$. With the previous preparations, we are ready to show the extensions of the ``Artstein-Avidan-Milman'' theorem for order preserving and order reversing mappings in {Section 7}. In {Section 8}, we are devoted to the extension of fundamental theorem of affine geometry to infinite dimensional case. With the help of the extended fundamental theorem of affine geometry, the representation theorem of fully order preserving mappings defined on the cone ${\rm semn}(X)$ of seminorms is shown in { Section 9}. The representation theorem of fully order preserving mappings defined on the cone ${\rm mink}(X)$ of Minkowski functionals is in {Section 10}. In the same sprit of {Section 9} and {Section 10}, some characterizations of fully order preserving mappings defined on sublinear functions will be provided in {Section 11}.

\section{Convex functions and their subdifferentials}
The definition and basic properties of subdifferential operator \cite{Ph} will be recalled and concepts concerning with order on cones of convex functions will be introduced. We conclude this section by providing characterizations
of minimum upper bound of convex functions.

The letter $X$ will always be a real Banach space and $X^*$ its dual. Recall that an extended real-valued convex function $f:X\rightarrow\R\cup\{\pm\infty\}$ is said to be proper, if  $f(x)>-\infty$ everywhere with ${\rm dom}(f)\coloneqq\{x\in X:f(x)<\infty\}\neq\emptyset$. By a cone $C$ of convex functions, we mean that it is convex and closed under multiplication of non-negative numbers. We use ${\rm conv}(X)$ to denote the cone of all proper convex lower semicontinuous (l.s.c., for short) functions on $X$; ${\rm aff}(X)$, the space of all continuous affine functions on $X$, i.e. ${\rm aff}(X)=\{\varphi+c: \varphi\in X^*,\;c\in\R\}$. 
For a subset $A\subseteq{\rm conv}(X)$, $\sup A$ stands for the convex function $f$ defined for $x\in X$ by $f(x)=\sup\limits_{h\in A}h(x)$; and (unless stated otherwise) $a\vee b$ for the function $g$ defined  by $g(x)=a(x)\vee b(x)\coloneqq\max\{a(x), b(x)\}$.

By a Minkowski functional $p$ we mean that it is an extended real-valued non-negative sublinear function on $X$, or equivalently, there is a convex set $D\subseteq X$ with $0\in D$ so that $p$ is generated by $D$, i.e. $p(x)=\inf\{\lambda>0: x\in\lambda D\},$ for all $x\in X$. Note that $p$ is l.s.c. if and only if the corresponding $D$ is closed.
For example, given a closed subset $K\subseteq X$, the indicator function $\delta_{K}:x\to \R\cup\{+\infty\}$ defined by $\delta_{K}(x)=0$ for $x\in K$ and $\delta_{K}(x)=+\infty$ for $x\not\in K$ is a convex proper function on $X$. In particular, when $K=0$, then $\delta_{0}=\delta_{\{0\}}$ is a l.s.c. Minkowski functional generated by the singleton $\{0\}$.

Given $f\in{\rm conv}(X)$, the subdifferential mapping $\partial f: X\rightarrow 2^{X^*}$ of $f$ is defined for $x\in X$ by
\[
\partial f(x)=\{\varphi\in X^*: \langle \varphi,y-x\rangle\leq f(y)-f(x),\;\;{\rm for\;all\;}y\in X\}.
\]
By the Br{\o}ndsted-Rockafellar theorem \cite{B-R} (see, also, \cite[Theorem 3.18]{Ph}), for every l.s.c. sublinear function $p$ defined on a Banach space $X$, we have
\[
\partial p(X)\coloneqq\bigcup_{x\in X}\partial p(x),
\]
is always nonempty $w^*$-closed and convex. Obverse that, for a l.s.c. sublinear function $p$, we have $\partial p(X)=\partial p(0)$. Indeed, for any $x^{*}\in \partial p(X)$, then, by the definition of $\partial p(X)$, there exists $x\in X$ so that
\[
\langle x^{*},y-x\rangle\leq p(y)-p(x),~\forall y\in X.
\]
Then, by the subadditivity of $p$, we have $\langle x^{*},z\rangle\leq p(z+x)-p(x)\leq p(z)$ for all $z\in X$. This entails that $\partial p(X)\subseteq \partial p(0)$.

Conversely, one can use the \emph{support function} to convert convex subsets to convex functions. More precisely, for a l.s.c. sublinear function $p$, we have that
\[
p(x)=\sigma_{\partial p(0)}(x)\coloneqq\sup_{x^*\in\partial p(0)}\langle x^*,x\rangle,\;\;\forall x\in X.
\]

For $G\subseteq{\rm conv}(X)$, we denote
\[
B_f(G)=\{g\in G: g\leq f\}.
\]
We simply write
\[
B_{f}=B_{f}({\rm conv}(X)), \;{\rm and\;} A_f=B_f({\rm aff}(X)).
\]

Since every extended real-valued l.s.c. convex function can be represented as the sup-envelope of a subset of ${\rm aff}(X)$ (see, for instance \cite{B-R}), we have
\[
\sup\limits_{h\in B_{f}}h=f=\sup\limits_{h\in A_{f}}h.
\]

Note that every affine function $u=\varphi+c$ can be regarded as a linear functional $\varphi\oplus c\in X^*\oplus\R$ restricted to the affine subspace $X\oplus\{1\}$ of $X\oplus\R$. We have the following property.
\begin{proposition}\label{sublinearconvex}
For every $f\in{\rm conv}(X)$ there is a l.s.c. sublinear function $p$ defined on $X\oplus\R$ so that $f=p(\cdot,1)$.
\end{proposition}
\begin{proof}
Note that $\sup\limits_{h\in A_{f}}h =f$, where $A_f=\{\varphi+c\leq f: \varphi\in X^*, c\in\R\}$. We define the sublinear function $p$ for $(x,r)\in X\oplus\R$ by
\[
p(x,r)=\sup\{\langle\varphi,x\rangle+rc: \varphi+c=\varphi\oplus c\in A_f\}.
\]
Then $f=p(\cdot,1)$.
\end{proof}
\begin{definition}
Let $S$ be a partially ordered set.
\begin{itemize}
\item[i)] $S$ is said to be an upper semi-lattice provided for all $a, b\in S$, $a\vee b\in S$.
\item[ii)] A subset $U$ of $S$ is called ordered provided for any $a, b\in U$ we have either $a\geq b$, or $b\geq a$.
\item[iii)] A subset $U$ of $S$ is called an upper bound provided for each $s\in S$ there is $u\in U$ so that $u\geq s$. An upper bound $U$ is called the minimum upper bound if for every upper bounded $V$ of $S$ we have $U\subseteq V$.
\item[iv)] A subset $U\subseteq S$ is said to be orderless if for every two different elements $a,b\in U$, we have $a\not\leq b$ and $b\not\leq a$.
\end{itemize}
\end{definition}

Before stating the next lemma, we now introduce the notation $[u,v]$. For $u$, $v\in {\rm conv}(X)$, the segment $[u,v]$ is the closed convex set $\{\lambda u+(1-\lambda)v:\lambda\in[0,1]\}$.

\begin{lemma}\label{upperbound}
Suppose that $X$ is a Banach space with $\dim(X)\geq 2$.
Let $u=\varphi+a, v=\psi+b\in{\rm aff}(X),$ and $h=u\vee v$. If $\varphi\neq\psi\in X^*$, then the segment
$[u,v]$ is the minimum upper bound of $A_h=\{t\in{\rm aff}(X), t\leq h\}.$
\end{lemma}
\begin{proof}
We first claim that $[u,v]$ is an upper bound of $A_h$.
For $\phi=x^*+c\in{\rm aff}(X)$ so that $\phi\leq h$, then there exists $\xi\in\R$ so that
$h-\xi$ is a sublinear function on $X$.
Since $\phi-\xi\leq h-\xi$,
then it follows that there exists $\lambda\in[0,1]$ satisfying $x^*=\lambda \varphi+(1-\lambda)\psi$.
Therefore, $\phi\leq \lambda u+(1-\lambda) v$.

Secondly, we claim that $[u,v]$ is the minimum upper bound of $A_{h}$. Indeed, for every upper bound $V$ of $A_{h}$, and $\lambda\in[0,1]$, there exists $\phi\in V$
so that $\lambda u+(1-\lambda)v\leq \phi$. Since $\phi$ and $\lambda u+(1-\lambda)v$ are affine functions so that
$\phi\leq h$ and $\lambda u+(1-\lambda)v\leq h=u\vee v$, then $\phi=\lambda u+(1-\lambda)v$. It follows that $[u,v]\subseteq V$.
Therefore, $[u,v]$ is the minimum upper bound of $A_{h}$.
\end{proof}

\section{Cones of convex functions and their sup-generating classes}

This section is devoted to the study of certain types of convex functions and their sup-generating class. Several concepts will be introduced and properties of these concepts will be discussed.

For a real Banach space $X$, we denote by $\mathfrak{C}(X)$ the set of all subcones $C$ of ${\rm conv}(X)$ satisfying $f\vee g\in C$ for all $f,g\in C$, and by $\mathfrak{C}_0(X)$ the set of the following specific classes of convex functions, which are  of our main interest. Precisely, $\big{\{}{\rm subl}(X),~{\rm mink}(X),~{\rm semn}(X)~{\rm conv}_{*}(X^{*})\big{\}}=\mathfrak{C}_{0}(X)\subseteq \mathfrak{C}(X)$, where\\

${\rm subl}(X)$ --- the cone of all l.s.c. sublinear functions on $X$;\\

${\rm mink}(X)$ --- the cone of all  l.s.c. Minkowski functionals on $X$;\\

${\rm semn}(X)$ --- the cone of all  l.s.c. seminorms on $X$;\\

${\rm conv}_{*}(X^*)$ --- the cone of all  $w^*$-l.s.c. convex functions defined on $X^*$.\\

For any $\Box\in\mathfrak{C}(X)$, ${\rm C}_\Box(X)$ presents the subcone of $\Box$ consisting of all continuous functions in $\Box$. For example, ${\rm C}_{\rm conv}(X)$ stands for the cone of all continuous convex functions on $X$.

Note that $\mathfrak{C}(X)$ is the set consisting of all subcones of ${\rm conv}(X)$, hence, for $K\in\mathfrak{C}(X)$, ${\rm C}_{K}$ is the subcone of all continuous convex functions in $K$. Denoted by $({\rm C}_{K},\varrho)$ equipped with the pointwise convergence topology, i.e. $f_\alpha\rightarrow f$ in the $\varrho$-topology provided
$f_\alpha(x)\rightarrow f(x)$ for all $x\in X$. Note that $\varrho$-topology is equivalent to the $w^*$-topology when $K=X^*$.



\begin{definition}\label{def3.1}
Let $C\in\mathfrak{C}(X)$. A subset $G\subseteq C$ is said to be a { sup-generating class} of $C$ provided for every $f\in C$ there is $B\subseteq G$ so that $\sup\limits_{h\in B}h=f$.
\end{definition}
To illustrate the definition of sup-generating class, we now provide a concrete example which shows that for certain subcone $C\in \mathfrak{C}(X)$, sup-generating class $G$ could be much more simpler than $C$ itself.
\begin{example}
In particular, let $C={\rm conv}(X)\in \mathfrak{C}(X)$, then, by an easy application of the Hahn-Banach theorem, we have that the set of all continuous affine functions ${\rm aff}(X)$ is a sup-generating class for ${\rm conv}(X)$.
\end{example}


There are many possibilities, but the following collections  of sup-generating classes  are of main interest to us:
\[
{\rm i)}\;\mathfrak{G}_{1}={\rm aff}(X),\; {\rm ii)}\;\mathfrak{G}_{2}=X^{*},\; {\rm iii)}\;\mathfrak{G}_{3}=X^{*+},\; {\rm iv)}\;\mathfrak{G}_{4}=|X^{*}|,
\]


where ${X^*}^+=\{\phi\vee0: \phi\in X^*\}\;\;{\rm and}\;\;\;{|X^*|}=\{\phi\vee -\phi:\phi\in X^*\}$.
Put
\[
\mathfrak{G}=\bigcup\limits_{j=1}^4\mathfrak{G}_j.
\]

\begin{definition}\label{def3.2}
Let $C\in\mathfrak{C}(X)$, and $G$ be a sup-generating class  of $C$.
\begin{itemize}
\item[i)] $G$  is called { pure}  provided it is closed under multiplication of non-negative numbers, and  $B_f(C)\subseteq G$ for all  $f\in G$.
\item[ii)] We say that  $G$ is { atomic} if $G\subseteq {\rm C}_{C}$ and $\varrho$-closed (i.e. for any net $\{f_{\alpha}\}\subseteq G$ if $f_{\alpha}\xrightarrow{\varrho} f$ for some $f\in{\rm C}_{C}$, then $f\in G$) such that for every $S\subseteq G$
with $\sup\limits_{h\in S}h=g\in G$ there is a monotone non-decreasing net $(g_\alpha)\subseteq S$ so that $\lim\limits_\alpha g_\alpha=g$ in the $\varrho$-topology.
\item[iii)] $G$  is said to be { perfect} if it is purely atomic.
\end{itemize}
\end{definition}
\begin{proposition}
Every perfect class $G$ of a cone $C\subseteq{\rm conv}(X)$ is the minimum element in all $\varrho$-closed sup-generating classes of $C$.
\end{proposition}
\begin{proof}
Let $F$ be a $\varrho$-closed sup-generating class of $C$, and let $f\in G$. Assume that $A\subseteq F$ satisfies $\sup\limits_{h\in A}h=f$. Then the purity and atomic of $G$ entails $f\in\overline{A}^{\varrho}\subseteq F.$ Therefore, $G\subseteq F$.
\end{proof}

\begin{proposition}\label{sup-generating class}
Let $X$ be a Banach space. Then
\begin{itemize}
\item[i)] ${\rm conv}(X)$ admits a perfect class ${\rm aff}(X)\in\mathfrak{G}_1(X)$;
\item[ii)] ${\rm subl}(X)$ admits a perfect class $X^*\in\mathfrak{G}_2(X)$;
\item[iii)] ${\rm mink}(X)$ admits a perfect class ${X^*}^+\in\mathfrak{G}_3(X)$;
\item[iv)] ${\rm semn}(X)$ admits a perfect class $|X^*|\in\mathfrak{G}_4(X)$.
\end{itemize}
\end{proposition}
\begin{proof}
{\rm i)}. Clearly, ${\rm aff}(X)$ is a $\varrho$-closed cone. Indeed, assume that $\{f_{\alpha}\}$ is a net of ${\rm aff}(X)$ so that $f_{\alpha}\xrightarrow{\varrho} f$ for some $f\in{\rm C}_{{\rm conv}(X)}$. Then, by the fact that $f_{\alpha}$ are affine for all $\alpha$, then $f$ is also affine.

Since  $\sup\limits_{h\in A_{f}}h=\sup\limits_{h\in B_{f}}h=f$ for every $f\in{\rm conv}(X)$,
${\rm aff}(X)$ is a sup-generating class of ${\rm conv}(X)$. To see that ${\rm aff}(X)$ is perfect, it suffices to note
\[
B_{f}=\{f+c: c\in\R^-\}\subseteq{\rm aff}(X), \;\;{\rm for\; all\;} f\in{\rm aff}(X).
\]
Hence, if $S=\{h_{\beta}\}_{\beta}\subseteq{\rm aff}(X)$ with $\sup\limits_{\beta}h_{\beta}=x^{*}+c$ for some $x^{*}\in X^{*}$ and $c\in \R$. By the affinity of $x^{*}+c$, it follows that $S\subset\{x^{*}+r:r\leq c\}$.
Without loss of generality, we assume that $h_{\beta}=x^{*}+r_{\beta}$ for all $\beta\in \triangle$ with
\[
x^{*}+c=\sup\limits_{\beta}\{x^{*}+r_{\beta}\}=x^{*}+\sup\limits_{\beta}r_{\beta}.
\]
Hence, $\sup\limits_{\beta}r_{\beta}=c$, which implies that there exists an increasing subnet $x^{*}+r_{\beta^{\prime}}$ so that $x^{*}+r_{\beta^{\prime}}\xrightarrow{\varrho}x^{*}+c$.
Thus, perfection of ${\rm aff}(X)$ follows.

{\rm ii)}. $\varrho$-closeness of the space $X^*$ is clear. Indeed, let $\{x^{*}_{\alpha}\}$ be net which $\varrho$-converges to some $f\in {\rm C}_{{\rm subl}(X)}$. Then, it follows that $f$ is a continuous linear functional, that is $f\in X^{*}$.
Note that for each $f\in{\rm subl}(X)$, there is a (unique) $w^*$-closed convex set
\[
\partial f(0)=\{x^*\in X^*: x^*\leq f\}
\]
so that $\sup\limits_{h\in \partial f(0)}h=f$. This and  orderless of $X^*$ yield purity and atomicity, hence, perfection of $X^*$.

{\rm iii)}. Obviously, ${X^*}^+$ is $\varrho$-closed.
Indeed, let $(x^{*+}_{\alpha})_{\alpha\in \bigtriangleup}\subseteq X^{*+}$ be a net so that $x^{*+}_{\alpha}$ is $\varrho$-convergent to some $f\in {\rm C}_{{\rm mink}(X)}$. Note that $\partial x^{*+}_{\alpha}(0)=[0,x^{*}_{\alpha}]$ for all $\alpha$, hence, the fact
\[
\lim\limits_{\alpha}\inf\limits_{\beta\geq \alpha}x^{*+}_{\beta}=f\in{\rm C}_{{\rm mink}(X)},
\]
yields that
\[
\bigcup\limits_{\alpha}\bigcap\limits_{\beta\geq \alpha}[0,x^{*}_{\beta}]\subseteq \partial f(0).
\]
By the Banach-Steinhauss theorem, it follows that $\partial f(0)$ is norm closed. Hence, by passing to a subnet, we can assume that $\{x^{*}_{\alpha}\}$ is norm bounded. Therefore, by the relative $w^{*}$-compactness of bounded subset of $X^{*}$, it follows that there exists a further subnet $\{x^{*}_{\alpha^{\prime}}\}$ of $\{x^{*}_{\alpha}\}$ so that $x^{*}_{\alpha^{\prime}}$ $w^{*}$-converges to $x^{*}_{0}$ for some $x^{*}_{0}\in X^{*}$. Consequently, $f=x^{*+}_{0}$.
To show that $X^{*+}$ is purely atomic, it suffices to note that for $g$, $h\in {\rm mink}(X)$, $g\leq h$ if and only if $\partial g(0)\subseteq \partial h(0)$. Then, any $g\in {\rm mink}(X)$ so that $g\leq x^{*_+}$, we have that $g=\lambda x^{*+}$ for some $\lambda\in[0,1]$.

{\rm iv)}. The $\varrho$-closeness of $|X^{*}|$ is analogous to the case of ${\rm iii)}$. Indeed, let $\{|x^{*}_{\alpha}|\}\subseteq |X^{*}|$ so that $|x^{*}_{\alpha}|\xrightarrow{\varrho} f$ for some $f\in {\rm C}_{{\rm semn}(X)}$. Then, by the Banach-Steinhauss theorem and the fact that
\[
\bigcup\limits_{\alpha}\bigcap\limits_{\beta\geq \alpha}[-x^{*}_{\beta},x^{*}_{\beta}]\subseteq \partial f(0),
\]
it follows that there exists a subnet $\{x^{*}_{\alpha^{\prime}}\}$ of $\{x^{*}_{\alpha}\}$ so that $x^{*}_{\alpha^{\prime}}$ $w^*$-converges to $x^{*}_{0}$. Therefore, $f=|x^{*}_{0}|$, which yields that $|X^{*}|$ is $\varrho$-closed. Hence, by noting that $B_{|x^{*}|}=\{g\in{\rm semn}(X):g\leq |x^{*}|\}=\{\lambda|x^{*}|:\lambda\in [0,1]\}$, the pure atomic of $|X^{*}|$ follows from a similar argument of ${\rm iii)}$.
\end{proof}


\begin{definition}
Let $C\in\mathfrak{C}(X)$ and $A\subseteq C$, then $A$ is said to be bounded if $\sup\limits_{h\in A} h(x)<+\infty$ for all $x\in X$.
\end{definition}
\begin{remark}
In particular, if $C=X^{*}$ and $A\subseteq X^{*}$, then by the Banach-Steinhauss theorem, we have that the boundedness defined as above coincides with the norm boundedness of $A$.
\end{remark}
\begin{definition}
A cone $C\in \mathfrak{C}(X)$ is said to be {(bounded) sup-complete} provided for every (bounded) set $A\subseteq C$ we have $\sup\limits_{h\in A}h\in C$ whenever $\sup\limits_{h\in A}h\neq\infty$.
\end{definition}
Given a cone $C\in \mathfrak{C}(X)$, we say that $\overline{C}$ is a (bounded) sup-completion of $C$, if $\overline{C}$ is the smallest cone such that $C\subseteq \overline{C}$ and $\overline{C}$ is (bounded) sup-complete.
\begin{proposition}
Every cone $C\in \mathfrak{C}(X)$ has a (bounded) sup-completion $\overline{C}$.
\end{proposition}
\begin{proof}
It suffices to put
\[
\overline{C}=\{\sup\limits_{h\in A} h\neq\infty:\emptyset\neq A\subseteq C\},
\]
\[
\big{(} \overline{C}=\{\sup\limits_{h\in A}h: \emptyset\neq A \;{\rm is\;a \;bounded \;set\;in}\; C\} \big{)}.
\]
We now prove the case that $C$ admits a sup-completion $\overline{C}$ and the same proof can be applied the case of bounded sup-completion.

Indeed, for any $\{f_{\alpha}\}_{\alpha}\subseteq \overline{C}$ with $\sup\limits_{\alpha}f_{\alpha}=g\not=\infty$, then, by the definition of $\overline{C}$ we have there exist $B_{f_{\alpha}}\subseteq C$ so that $\sup\limits_{h\in B_{f_{\alpha}}}h=f_{\alpha}$ for all $\alpha$. Let $A=\bigcup\limits_{\alpha}B_{f_{\alpha}}$, we have that $A\subseteq C$ and $g=\sup\limits_{h\in A}h$, which entails that $g\in \overline{C}$.
\end{proof}

\begin{proposition}\label{sup-complete}
\begin{itemize}
\item[i)] If the cone $C$ is one of the following four cones
\[
\big{\{} {\rm conv}(X),~ {\rm subl}(X),~ {\rm mink}(X),~ {\rm semn}(X) \big{\}}
\]
then $C$ is sup-complete.
\item[ii)] If the cone $C$ is one of the following four cones
\[
\big{\{} {\rm C}_{{\rm conv}(X)},~ {\rm C}_{{\rm subl}(X)},~ {\rm C}_{{\rm mink}(X)},~ {\rm C}_{{\rm semn}(X)} \big{\}}
\]
then $C$ is bounded sup-complete.
\end{itemize}

\end{proposition}
\begin{proof}
{\rm i)}. If $C$ is one of the four cones as above, then for any family of functions $\{f_{\alpha}\}_{\alpha}\subseteq C$, if $\sup\limits_{\alpha}f_{\alpha}$ is proper, that is, $\sup\limits_{\alpha}f\neq \infty$. Then, $\sup\limits_{\alpha} f_{\alpha}$ must be a l.s.c. convex function of same type, that is $\sup\limits_{\alpha}f_{\alpha}\in C$.

{\rm ii)}. It is analogous to the proof of {\rm ii)}. Indeed, if $C$ is one of the four cones as {\rm ii)}, then, for any bounded $A\subseteq C$, we have that $\sup\limits_{h\in A}h(x)<+\infty$ for all $x\in X$. This entails that $\sup\limits_{h\in A}h$ is a continuous convex function of same type.
\end{proof}

\section{Fully order preserving mappings on cones of convex functions}
In this section, we shall show that every fully order preserving mapping defined on a bounded sup-complete and $\varrho$-closed cone is bounded $\varrho$-continuous. To begin with, we recall the definition of fully order preserving and reversing mappings.

\begin{definition}
Let $(P,\leq)$ be a partially ordered set. A mapping $T:P\rightarrow P$ is {fully order preserving (resp. reversing)} if it is a bijection satisfying $f\geq g\Longleftrightarrow Tf\geq Tg$ (resp. $f\geq g\Longleftrightarrow Tf\leq Tg$).

\end{definition}
\begin{example}
Let $\mathscr C_{*}(X^*)$ (resp. $\mathscr C_{0,*}(X^*)$) be the collection of all nonempty $w^*$-closed convex sets (resp. containing the origin $0$) of $X^*$. Then $\mathcal D=\partial p(0): {\rm subl}(X)\rightarrow\mathscr C_*(X^*)$ (resp. ${\rm mink}(X)\rightarrow\mathscr C_{0,*}(X^*)$) is a fully order preserving mapping, where $\partial p$ denotes the subdifferential of $p$.
Conversely, $\mathcal S\coloneqq\mathcal D^{-1}: \mathscr C_*(X^*)\rightarrow{\rm subl}(X)$ defined for $C\subseteq \mathscr C_*(X^*)$ by
\[
(\mathcal D^{-1}(C))(x)=\sup_{c\in C}\langle c,x\rangle=\sigma_{C}(x), \;x\in X
\]
is also fully order preserving. Therefore, for every fully order preserving mapping $T: {\rm subl}(X)\rightarrow{\rm subl}(X)$,
\begin{equation}\label{composition1}
F\coloneqq\mathcal DT\mathcal D^{-1}:\mathscr C_*(X^*) \rightarrow\mathscr C_*(X^*)
\end{equation}
is fully order preserving.
\end{example}

The following lemma  was motivated by \cite[Lemma 2]{A-A-M09}.

\begin{lemma}\label{continuity}
Let $C\in{\mathfrak C}(X)$ be a cone, $T: C\rightarrow C $ be a fully order preserving mapping, and $(f_\alpha)\subseteq C$ be a set.
\begin{itemize}
\item[i)] If both $\sup\limits_\alpha f_\alpha\;{\rm and\;}\sup\limits_\alpha Tf_\alpha$ belong to $C$, then
\[
T\sup_\alpha f_\alpha=\sup_\alpha Tf_\alpha.
\]
\item[ii)] If, in addition,  $(f_\alpha)$ is a net, and there is $\eta$ so that for all $\zeta\geq\eta$
\[
\limsup_\alpha f_\alpha,\; \limsup_{\alpha}Tf_\alpha, \;\sup_{\alpha\geq\zeta}f_\alpha\in C,
\]
then
\begin{equation}\label{limits}
\limsup_\alpha Tf_\alpha=T\limsup_\alpha f_\alpha.
\end{equation}
\end{itemize}
\end{lemma}
\begin{proof}
{\rm i)}. Let $\sup\limits_\alpha f_\alpha=f$ and $\sup\limits_\alpha Tf_\alpha=f_T.$ Then $f,f_T\in C$. Since $T$ is surjective, then it entails that there is $g\in C$ so that $Tg=f_T$. Since $T$ is fully order preserving,
$f\geq f_\alpha$ implies $Tf\geq Tf_\alpha$, and further, $Tf\geq Tg$.  Consequently, $f\geq g$. Conversely, $Tg\geq Tf_\alpha$  yields $g\geq f_\alpha$ for all $\alpha$. Thus, $g\geq f$. We have shown $f=g$. Equivalently, $Tf=Tg$.

{\rm ii)}. Let $f=\limsup\limits_\alpha f_\alpha$ and $g_\alpha=\sup\limits_{\xi\geq\alpha} f_\xi$ for all $\alpha\geq\eta$. Then the assumption says $f,g_\alpha\in C$ with $g_\alpha\geq f$, and $f=\lim_\alpha g_\alpha.$  According to {\rm i)},
\begin{equation}\label{comparable}
Tg_\alpha=T\sup_{\xi\geq\alpha}f_\xi=\sup_{\xi\geq\alpha}Tf_\xi\geq Tf,
\end{equation}
and which further deduces
\[
\lim\limits_{\alpha}Tg_\alpha=\limsup\limits_{\alpha}Tf_\alpha\in C.
\]
Note that $T$ is surjective, then it yields that there is $g\in C$ so that  $\lim\limits_{\alpha}Tg_\alpha=Tg$. This and \eqref{comparable} entail  $Tg\geq Tf$. On the other hand, the isotonicity (i.e. order preserving) of $T$ and
non-increasing monotonicity of $(g_\alpha)$ deduce that $g_\alpha\geq g$ for all $\alpha$. Consequently, $f=\limsup\limits_\alpha f_\alpha=\lim\limits_\alpha g_\alpha\geq g.$
This and the isotonicity of $T$ again lead to $Tf\geq Tg$. We finish the proof by noting that $Tf=Tg$ is equivalent to \eqref{limits}.
\end{proof}
\begin{theorem}\label{continuity theorem}
Suppose that $C\subseteq{\rm C_{\rm conv}}(X)$ is a bounded sup-complete and $\varrho$-closed. If $T: C\rightarrow C$ is fully preserving. Then $T$ is bounded $\varrho$-continuous, i.e.
for every bounded net $(f_\alpha)\subseteq G$ $\varrho$-convergent to $f\in G$, we have $Tf=\lim\limits_{\alpha} Tf_{\alpha}$.
\end{theorem}
\begin{proof}
Let $(f_\alpha)\subseteq C$ be a bounded net which is $\varrho$-convergent to $f\in C$. Since $C$ is a bounded sup-complete, it follows that both $g\coloneqq\sup\limits_\alpha f_\alpha$ and $g_\alpha\coloneqq\sup\limits_{\xi\geq\alpha}f_\xi$ belong to $C$.
By Lemma \ref{continuity} {\rm i)}, we have that
\begin{equation}\label{compare inequality}
Tf\leq Tg_\alpha=\sup_{\xi\geq\alpha}Tf_{\xi}\leq Tg\in C, \;\;\;{\rm for\;all\;}\alpha.
\end{equation}
Note that $Tg\in C$ and hence $Tg$ is continuous, then, by the sup-completeness again, it follows that $Tg_{\alpha}=\sup\limits_{\xi\geq \alpha}Tf_{\xi}\in C$ for all $\alpha$.

The non-increasing monotonicity of $\{Tg_\alpha\}_{\alpha}$ implies that the function
\[
g_{T}(x)=\lim\limits_{\alpha}Tg_{\alpha}(x),~x\in X,
\]
is a extended real-valued convex function. It follows from \eqref{compare inequality} that
\begin{equation}\label{dominated}
Tf\leq g_T\leq\sup_{\xi\geq\alpha}Tf_\xi\in C \;\;\;{\rm for\;all\;}\alpha.
\end{equation}
By the fact that a convex function which is locally bounded above then it must locally Lipschitzian (see \cite[p.39 Proposition 3.3]{Ph}). Thus, by \eqref{dominated}, we have that $g_T$ is continuous. Since $C$ is $\varrho$-closed, then we obtain that $g_{T}\in C$. According to Lemma \ref{continuity} {\rm ii)},
\[
Tf=T\lim_\alpha\sup_{\xi\geq\alpha} f_\alpha=\lim_\alpha\sup_{\xi\geq\alpha}Tf_\alpha.
\]
It is easy to observe that for every subnet $\{f_{\alpha^{\prime}}\}$ of $\{f_{\alpha}\}$ we also have
\begin{equation}\label{subnet}
Tf=\lim\limits_{\beta^{\prime}}\sup_{\xi^{\prime}\geq\beta^{\prime}}Tf_{\xi^{\prime}}.
\end{equation}
Therefore, $\lim\limits_{\alpha}Tf_{\alpha}$ exists and equals to $Tf$. Indeed, if there exist $\varepsilon_{0}>0$, $x_{0}\in X$ and a subnet $\{f_{\beta^{\prime}}\}_{\beta^{\prime}}$ of $\{f_{\alpha}\}$ so that
\[
\big{|}(Tf_{\beta^{\prime}})(x_{0})-(Tf)(x_{0})\big{|}>\varepsilon_{0}.
\]
By passing to a further subnet, we can assume, without loss of generality, that
\begin{equation}\label{strict inequality}
(Tf_{\beta^{\prime}})(x_{0})<(Tf)(x_{0})-\varepsilon_{0}.
\end{equation}
Hence, by \eqref{strict inequality}, it follows that
\[
\lim\limits_{\beta^{\prime}}\sup\limits_{\xi^{\prime}\geq \beta^{\prime}}Tf_{\xi^{\prime}}(x_{0})\leq Tf(x_{0})-\varepsilon_{0}<Tf(x_{0}),
\]
which contradicts to \eqref{subnet}.
\end{proof}

\begin{corollary}\label{bounded rho continuity}
Suppose that the cone $C$ is one of the following eight cones
\[
{\rm conv}(X), {\rm C}_{\rm conv}(X), {\rm subl}(X), {\rm C}_{\rm subl}(X), {\rm mink}(X), {\rm C}_{\rm mink}(X), {\rm semn}(X), {\rm C}_{\rm semn}(X).
\]
Then every fully order preserving mapping is bounded $\varrho$-continuous.
\end{corollary}
\begin{proof}
By Proposition \ref{sup-complete} {\rm ii)}, we have that
\[
C\in\{ {\rm C}_{\rm conv}(X),~ {\rm C}_{\rm subl}(X),~ {\rm C}_{\rm mink}(X),~ {\rm C}_{\rm semn}(X)\}
\]
is bounded sup-complete and the $\varrho$-closeness of $C$ is obvious, then, by Theorem \ref{continuity theorem}, it suffices to show the conclusion holds for
\[
C\in\{{\rm conv}(X), {\rm subl}(X), {\rm mink}(X),  {\rm semn}(X)\}.
\]
Given such a cone $C$ and note that the definition of $\varrho$-closeness only invokes the subset ${\rm C}_{C}$ of $C$. Then, for a fully order preserving mapping $T:C\rightarrow C$, according to Theorem \ref{continuity theorem} again, it suffices to prove that the restriction $T|_{{\rm C}_C}$ maps each continuous function $f\in {\rm C}_{C}$ into a continuous one.

Suppose, to the contrary, that there is $f\in {\rm C}_{C}$ but $Tf\notin {\rm C}_{C}$. Since $Tf$ is l.s.c. and proper, there exists $0\neq x_{0}\in X$ so that

case {\rm i)}. $(Tf)(x_{0})=\infty$ if $C={\rm conv}(X)$;

case {\rm ii)}. $(Tf)(\lambda x_{0})=\infty$ for all $\lambda>0$, if $C={\rm subl}(X)$, or,  ${\rm mink}(X)$;

case {\rm iii)}. $(Tf)(\lambda x_{0})=\infty$ for all $\lambda\neq0$, if $C={\rm semn}(X)$.

Let
\[
D=\left\{\begin{array}{ccc}
                    \{x_{0}\}~, &{\rm case\; i)} ; \\
                    \{\lambda x_{0}: \lambda\geq 0\}~,  &{\rm  case\; ii)} ;\\
                      \{\lambda x_{0}: \lambda\in\mathbb R\}~, &{\rm  case\; iii)}.
                  \end{array}
\right.
\]
For case {\rm i)} the convex function $\delta_D=\delta_{x_{0}}$ is a maximal element of $C={\rm conv}(X)$. Since $T$ is fully order preserving, there is a $q\in X$ so that $\delta_{x_{0}}=T\delta_{q}$.
Since $f$ is continuous, then $g\coloneqq f\vee\delta_{q}\in C$. Therefore, by Lemma \ref{continuity} {\rm i)},
\[
\infty=Tf\vee T\delta_{q}=Tg\in C,
\]
and this is a contradiction.

For case {\rm ii)} the convex function $\delta_D$ is a maximal element but not the maximum in $C\setminus\{\delta_0\}$, $T\delta_D$ is again  maximal in $C\setminus\{\delta_0\}$. Therefore, there is $0\neq q\in X$ so that $T\delta_E=\delta_D$, where $E=\{\lambda q:\lambda\geq0\}$.
Thus,
\[
\delta_0=Tf\vee T\delta_E=T(f\vee\delta_E)
\]
is the maximum element in $C$. But $f\vee\delta_E$ is not the maximum element of $C$. This contradicts to that $T$ is fully order preserving.

Analogously, we can show that case {\rm iii)} cannot happen.
\end{proof}

\section{Fully order preserving mappings restricted to sup-generating classes}
We will prove in this section that fully order preserving mappings behave nicely when restricted to certain sup-generating class.

Let $G$ be a sup-generating class of $C$, and a mapping $T:G\to G$. We say that $T$ is affine if for $u$, $v\in G$ with $\lambda\in \R$ so that $\lambda u+(1-\lambda)v$, $T(u)$, $T(v)$ and $\lambda T(u)+\lambda T(v)$ are all in $G$, then $T\big{(}\lambda u+(1-\lambda)v \big{)}=\lambda T(u)+(1-\lambda) T(v)$. Similarly, we can define a mapping $T:G\to G$ is additive (resp. homogeneous) in the same way.

The following results state that for a fully order preserving mapping $T$ defined on a cone $C$,  many nice properties of $T$ on a sup-generating class of $C$ can be passed on to the whole cone $C$.
\begin{lemma}\label{generating}
Let $C\in\mathfrak{C}(X)$, and $G$ be a  sup-generating class of $C$. Suppose that $T:C\rightarrow C$ is a fully order preserving mapping. Then
\begin{itemize}
\item[i)] $T$ is affine on $C$ if and only if ${\rm for\; all\;}\lambda\in\R\;{\rm and}\; u,v\in G,$
\[
T(\lambda u+(1-\lambda)v)=\lambda Tu+(1-\lambda)Tv;
\]
\item[ii)] $T$ is additive on $C$ if and only if
\[
T(u+v)=Tu+Tv\;\;{\rm for\; all}\; u,v\in G;
\]
\item[iii)] $T$ is positively homogenous  on $C$ if and only if  it is positively homogenous on $G$.
\end{itemize}
\end{lemma}

\begin{proof}
Given $f,g\in C$ and $0\leq \lambda \leq 1$, let $(f_{\alpha})$, $(g_{\beta})\subseteq G$ such that $f=\sup\limits_{\alpha}f_{\alpha}$ and $g=\sup\limits_{\beta}g_{\beta}$. Then
\[
\begin{split}
T\big{(}\lambda f+(1-\lambda)g\big{)}&=T\big{(}\lambda\sup\limits_{\alpha}f_{\alpha}+(1-\lambda)\sup\limits_{\beta}g_{\beta}
\big{)}\\
                         &=T\big{(}\sup\limits_{\alpha,~\beta}(\lambda f_{\alpha}+(1-\lambda) g_{\beta})\big{)}\\
                         &=\sup\limits_{\alpha,~\beta}T\big{(}\lambda f_{\alpha}+(1-\lambda) g_{\beta}\big{)}\\
                         &=\sup\limits_{\alpha,~\beta}\big{(}\lambda T f_{\alpha}+(1-\lambda) T g_{\beta} \big{)}\\
                         &=\sup\limits_{\alpha}\lambda T(f_{\alpha})+\sup\limits_{\beta} (1-\lambda)T(g_{\beta})\\
                         &=\lambda T(f)+(1-\lambda)T(g),
\end{split}
\]
which completes the proof of {\rm i)}.
The proof of {\rm ii)} and {\rm iii)} are in the same spirt of the case {\rm i)}.
\end{proof}

The following proposition states that fully order preserving mappings defined on a class $C$ of convex functions can be reduced into a nice subclass of $C$.

\begin{proposition}\label{restriction}
Let $C\in\mathfrak{C}(X)$, and $T: C\rightarrow C$ be a fully order preserving mapping.
Assume that $C$ admits a perfect class $G$. Then $T|_G: G\rightarrow G$ is again a fully order preserving.
\end{proposition}

\begin{proof}
Clearly, it suffices to show $TG=G$. Since $G$ is a sup-generating class and $T$ is a fully order preserving mapping, $TG$ is again a sup-generating class of $C$. Thus, given $g\in G$, we get $\sup\limits_{h\in B_g(TG)}h=g$. Purity of $G$ deduces $B_g(TG)\subseteq G$, and further, there is a monotone non-decreasing net $\{g_\alpha\}_{\alpha}\subseteq TG\cap G$ so that $\lim\limits_{\alpha} g_\alpha=\limsup\limits_{\alpha} g_\alpha=g.$ Let $Th_\alpha=g_\alpha (\leq g)$ for some $h_\alpha\in G$ and for each $\alpha$. By the fact that $T$ is order preserving, then it yields that $\{h_{\alpha}\}$ is also monotone non-decreasing. Thus, there is a convex function $h$ such that $\lim\limits_{\alpha} h_{\alpha}=h$. Now the $\varrho$-closeness of $G$ implies $h\in G$. By Lemma \ref{continuity} {\rm ii)},
\[
g=\lim_\alpha Th_\alpha=\limsup_\alpha Th_\alpha=T\limsup_\alpha h_\alpha=Th.
\]
We have shown $G\subseteq TG$.
Note that $T^{-1}$ is also fully order preserving. By a similar discussion on $T^{-1}$, we get $T^{-1}(G)\supseteq G.$
Hence, $TG=G$. 
\end{proof}

\section{Fully order preserving mappings on certain classes of cones}
In this section, we shall show that every fully order preserving mapping defined on some specific classes of cones including ${\rm conv}(X)$ and ${\rm subl}(X)$ is affine and $\varrho$-continuous on its sup-generating class.

Recall that for two convex functions $u$, $v$, denoted by $[u,v]$ the interval generated by $u$ and $v$, that is $[u,v]=\{\lambda u+(1-\lambda)v:\lambda\in[0,1]\}$. A subset $M$ of convex functions is said to be ordered (resp. orderless)
if for every $u$, $v\in M$, then either $u\leq v$ or $v\leq u$ (resp. $u$ and $v$ are incomparable).
\begin{theorem}\label{thm}
Suppose that $C\in\mathfrak{C}(X)$ is a bounded sup-complete, which consists of continuous convex functions, and $T:C\rightarrow C$ is a fully order preserving mapping.
Assume that $C$ admits a $\varrho$-closed pure class  $G\subseteq{\rm aff}(X)=X^*\oplus\mathbb R$. Then
\begin{itemize}
\item[i)] $G$ is perfect and $T: C\rightarrow C$ is affine with $TG=G$;
\item[ii)] $T$ is $\varrho$-continuous on $G$;
\item[iii)] In particular, if $G\subseteq X^*\oplus\mathbb R $ is a $w^*$-closed linear subspace, then $S=T-T(0)$ is a $w^*$-continuous linear operator on $G$.
\end{itemize}
\end{theorem}

\begin{proof}
{\rm i)}. To prove that $G$ is a perfect class of $C$, it suffices to show that $G$ is atomic. Assume $S\subseteq G$ so that $\sup\limits_{h\in S}h=g\in G\subseteq {\rm aff(X)}$. Without loss of generality, $g=x^*+r$ for some $x^*\in X^*, r\in \mathbb{R}$. Since $\sup\limits_{h\in S}h=g$, then $S$ must consist of all elements of the form $x^*+r_{\alpha}$ for some $r_{\alpha}\leq r$, and hence $\sup{r_{\alpha}}=r$.
Consequently, let $f_{\beta}=x^*+r_{\beta}\subseteq S$ such that $f_{\beta}\nearrow g$ in the $\varrho$-topology.
Hence, $G$ is perfect. Applying with the proof of Proposition \ref{restriction}, $TG=G$.


In order to show the affinity of $T$, due to the fundamental theorem of affine geometry and Lemma \ref{generating}, we only need to prove that $T[u,v]=[Tu,Tv]$  for all $u,v\in G$. Note that $G\subseteq {\rm aff}(X)$, then a segment $[u,v]$ is ordered (resp. orderless) if and only if  $\{u,v\}$  is ordered (resp. orderless).
Now, fix two different elements $u=\varphi+c, v=\psi+d\in G$, and let $h=u\vee v$.

If $\{u,v\}$ is ordered, say $u\geq v$, then $h=u=\varphi+c,v=\phi+d$ with $c\geq d$. Therefore,
\[
T[u,v]\subseteq B_{Tu}(C)\cap\{w\in C:w\geq Tv\}=[Tu,Tv].
\]
Since $T^{-1}$ is also fully order preserving, by a similar discussion but on $T^{-1}$, we have $T^{-1}[Tu,Tv]\subseteq[u,v]$. Thus, $T[u,v]=[Tu,Tv]$.

It remains to the case when $\{u,v\}$ is orderless, then $\varphi\neq\psi$, $[u,v]$ and $\{Tu,Tv\}$ are orderless. By Lemma \ref{upperbound}, $[u,v]$ is the minimum upper bound of $B_h(G)$.
Since $T$ is fully order preserving, $T[u,v]$ is an orderless set which is the minimum upper bound of $TB_h(G)=B_{Th}(G)=B_{Tu\vee Tv}(G)$.
On the other hand, since $[Tu,Tv]$ is also the   minimum upper bound  of  $B_{Tu\vee Tv}(G)$, we obtain $T[u,v]=[Tu,Tv]$. Thus, we have shown that $T|_G$ is affine.

{\rm ii)}. By Theorem \ref{continuity theorem}, we have shown that $T|_{G}:G\to G$ is bounded $\varrho$-continuous. Since
$G\subseteq {\rm aff}(X)=X^{*}\oplus \R$, then the $\varrho$-topology can be identified by the $w^{*}$-topology by noting Proposition \ref{sublinearconvex}. By {\rm i)}, we have that $T$ is affine, then $S=T-T(0)$ is linear, which is bounded $\varrho$-continuous. Thanks to the Grothendieck's dual characterization of completeness (see \cite[p.149 Corollary 2]{SchH}), we know that $S:G\to G$ is $\varrho$-continuous, then so does $T=S+T(0)$.

{\rm iii)}. It follows from {\rm i)} and {\rm ii)} directly.
%
%
%
\end{proof}
\begin{corollary}\label{sublinear}
Suppose that $T:{\rm conv}(X)\rightarrow {\rm conv}(X)$ is a fully order preserving mapping.
Then,
\begin{itemize}
\item[i)] the restriction $T|_{{\rm C}_{{\rm conv}(X)}}: {\rm C}_{{\rm conv}(X)}\to {\rm C}_{{\rm conv}(X)}$ is bounded $\varrho$-continuous affine mapping with $T\big{(}{\rm aff}(X)\big{)}={\rm aff}(X)$;
\item[ii)] $S=T|_G-T(0)$ is a $\varrho$-continuous linear operator on $X^*\oplus\mathbb R$.
\end{itemize}
\end{corollary}
\begin{proof}
By Theorem \ref{thm}, it suffices to note that for each continuous convex function $f$, then $Tf$ is also continuous (see, the proof of Corollary \ref{bounded rho continuity}), and note that ${\rm C}_{{\rm conv}(X)}$ is bounded sup-completed and $\varrho$-closed, which admits the perfect class ${\rm aff}(X)=X^*\oplus\R$.
\end{proof}

\begin{corollary}
Suppose $T:{\rm subl}(X)\rightarrow {\rm subl}(X)$ is a fully order preserving mapping.
Then,
\begin{itemize}
\item[i)] the restriction $T|_{{\rm C}_{{\rm subl}(X)}}: {\rm C}_{{\rm subl}(X)}\to {\rm C}_{{\rm subl}(X)}$ is bounded $\varrho$-continuous affine mapping with $TX^*=X^*$;
\item[ii)] $S= T|_G-T(0)$ is a $w^*$-continuous linear mapping on $X^*$.
\end{itemize}
\end{corollary}

\begin{proof}
The proof is complete analogous to the proof of Corollary \ref{sublinear}. By Theorem \ref{thm} it suffices to prove that $T$ maps ${\rm C}_{{\rm subl}(X)}$ onto itself. And this follows from the proof of Corollary \ref{bounded rho continuity} directly. Furthermore, by Proposition \ref{sup-generating class}, we have that $X^{*}$ is a $\varrho$-closed (with respect to ${\rm C}_{{\rm subl}(X)}$) perfect class of ${\rm C}_{{\rm subl}(X)}$, then, by Theorem \ref{thm}, it follows that $T(X^{*})=X^{*}$ such that $T$ is $w^{*}$-to-$w^{*}$ continuous.
\end{proof}

\begin{theorem}
Let $C\in\mathfrak{C}(X)$ and $T:C\rightarrow C$ be a fully order preserving mapping.
Assume that $C$ admits a $\varrho$-closed pure class $G\in\mathfrak{G}_{3}$, i.e. $G\subseteq{X^*}^+$. Then $G$ is perfect such that $TG=G$ and $T(0)=0$.
\end{theorem}
\begin{proof}
Note for every $\varphi\in X^*$,
\begin{equation}\label{structure}
B_{\varphi^+}(G)=\{\lambda\varphi^+: 0\leq\lambda\leq1\}.
\end{equation}
Since $G\subseteq{X^*}^+$ is $\varrho$-closed and pure, it is not difficult to see that $G$ is atomic.  By Proposition \ref{restriction}, $T|_G$ is fully order preserving with $TG=G$. To see that $T(0)=0$, it suffices to note that if $C$ admits a perfect class $G\subseteq \mathfrak{G}_{3}=X^{*+}$, then $C$ is a subclass of Minkowski functionals (i.e. ${\rm mink}(X)$) and $0$ is the smallest element in ${\rm mink}(X)$. Then, by the fact that $T$ is a fully order preserving mapping, we have $T(0)=0$.
\end{proof}

\section{Infinite dimensional version of the ``Artstein-Avidan-Milman'' theorem }
In this section, we shall show an exact infinite dimensional version of the ``Artstein-Avidan-Milman'' theorem for fully order reversing mappings.
We begin with the following property.
\begin{proposition}\label{infiniteversion}
Suppose that $X$ and $Y$ are two Banach spaces. Then the following statements are equivalent:
\begin{itemize}
\item[i)] there exists a fully order preserving mapping ${\rm conv}(X)\rightarrow{\rm conv}(Y)$;
\item[ii)] $X$ is isomorphic to $Y$.
\end{itemize}

\end{proposition}
\begin{proof}
Clearly, it suffices to show ${\rm i)} \Longrightarrow {\rm ii)}$.

Suppose that $T:{\rm conv}(X)\rightarrow{\rm conv}(Y)$ is a fully order preserving mapping.
Note that ${\rm aff}(X)$ (resp. ${\rm aff}(Y)$) is the perfect sup-generating class of  ${\rm conv}(X)$ (resp. ${\rm conv}(Y)$). Since $T$ is fully order preserving, it maps ${\rm aff}(X)$ onto ${\rm aff}(Y)$. Indeed, given $\phi+r\in{\rm aff}(X)$, since $B_{\phi+r}\coloneqq\{f\in{\rm conv}(Y): f\leq\phi+r\}=\{\phi+s: r\geq s\in\mathbb R\}$ is well-ordered, then  $T(B_{\phi+r})$ is again ordered. The only possible case is $T({\phi+r})\in{\rm aff}(Y)$. Since $T^{-1}$ is also fully order preserving, $T_{{\rm aff}(X)}:{\rm aff}(X)\rightarrow {\rm aff}(Y)$ (the restriction of $T$ to  ${\rm aff}(X)$) is again fully order preserving. By an argument similar to the proof of Theorem \ref{thm},
\[
T_{{\rm aff}(X)}: {\rm aff}(X)\rightarrow {\rm aff}(Y)
\]
is a $\varrho$-to-$\varrho$-continuous  affine isomorphism. Note
\[
{\rm aff}(X)=X^*\oplus\mathbb R,\;\;{\rm aff}(Y)=Y^*\oplus\mathbb R,
\]
and note the $\varrho$-to-$\varrho$-continuity is equivalent to the $w^*$-to-$w^*$ continuity.
Then there is a $w^*$-to-$w^*$ continuous isomorphism $V: X^*\rightarrow Y^*$, which implies that $V$ is an adjoint operator. Hence, there exists an isomorphism $U:Y\to X$ so that $U^{*}=V$.
\end{proof}
The following theorem is an analogous, or, a slight generalization of ``Artstein-Avidan-Milman-Iusem-Reem-Svaiter'' representation theorem of fully order preserving mappings. However, our proof is simpler.
\begin{theorem}\label{presrving}
Suppose that $X$ and $Y$ are two Banach spaces. Then for every fully order preserving mapping $T:{\rm conv}(X)\rightarrow{\rm conv}(Y)$, there exists an isomorphism $U:Y\rightarrow X$, $x_0\in X$,
$\varphi\in Y^*$, $\alpha>0$ and $r_0\in\mathbb R$ so that
\[
T(f)(y)=\alpha f(Uy+x_0)+\langle\varphi,y\rangle+r_0,\;\;f\in{\rm conv}(X),\;y\in Y.
\]
\end{theorem}
\begin{proof}
Suppose that $T:{\rm conv}(X)\rightarrow{\rm conv}(Y)$ is a fully order preserving mapping. Then by the proof of Proposition \ref{infiniteversion}, we obtain that
\begin{itemize}
\item[i)] $T$ maps ${\rm aff}(X)$ onto ${\rm aff}(Y)$;
\item[ii)] $S=T-T(0)$ is $\varrho$-to-$\varrho$ continuous fully order preserving linear mapping with $S(\mathbb R)=\mathbb R,$ and with $S(1)\coloneqq\alpha>0$.
\end{itemize}

Since $X^*\subseteq {\rm aff}(X)$ is orderless and $w^*$-closed, $S(X^*)$ is again an orderless $w^*$-closed hyperplane of ${\rm aff}(Y)=Y^*\oplus\mathbb R.$

Therefore, there exists $y_0+s_0\in Y\oplus\mathbb R$ so that
\[
S(X^*)=\{y^*+s\in Y^*\oplus\mathbb R: \langle y^*,y_0\rangle+sr_0=0\}.
\]
We claim $s_0\neq0$. Otherwise, we have $\mathbb R\subseteq S(X^*)$. This contradicts to that $S(X^*)$ is orderless.
Thus,
\[
S(X^*)=\{y^*-\langle y^*,y_0\rangle/s_0: y^*\in Y^*\}.
\]
Let $P: Y^*\oplus\mathbb{R}\rightarrow Y^*$ be the projection along $\{0\}\oplus\mathbb R$.
Thus,
\[
W(x^*)=PS(x^*),\;\;x^*\in X^*
\]
is a linear bijection from $X^*$ onto $Y^*$. Indeed, let $x^*_1 \not=x^*_2$, then $S(x^*_1)$ must be incomparable with $S(x^*_2)$. Hence, it follows that $W(x^*_1)\not =W(x^*_2)$. For any $y^*\in Y^*$, there exists $x^*+r\in X^* \oplus \mathbb{R}$ so that $S(x^*+r)=y^*$, which yields that $W$ is surjective. The $w^*$-to-$w^*$ continuity of $W$ follows from Lemma \ref{continuity} directly.
Therefore, there is isomorphism $V:Y\rightarrow X$ so that $V^*=W$.

Given $f\in {\rm conv}(X)$, let
\[
B_{f}\big{(}{\rm aff}(X)\big{)}=\{x^*+r\leq f: x^*+r\in{\rm aff}(X)\}.
\]
Then it follows from Lemma \ref{continuity} for all $y\in Y$
\begin{eqnarray}
(Sf)(y)&=&\big{(}S\sup\limits_{h\in B_{f}({\rm aff}(X))}h\big{)}(y)=\sup\limits_{h\in B_f({\rm aff}(X))}(Sh)(y)\\\nonumber
&=&\sup\{\langle Sx^*,y\rangle+\alpha r: x^*+r\leq f\}\\\nonumber
&=&\sup\{\langle Wx^*,y-y_0/s_0\rangle+\alpha r: x^*+r\leq f\}\\\nonumber
&=&\sup\{\langle x^*,V(y-y_0/s_0)\rangle+\alpha r: x^*+r\leq f\}\\\nonumber
&=&\alpha\sup\{\langle x^*,\alpha^{-1}V(y-y_0/s_0)\rangle+r: x^*+r\leq f\}\\\nonumber
&=&\alpha f(Uy+x_0),
\end{eqnarray}
where $U=\alpha^{-1} V$ and $x_0=-U\big{(}\frac{y_{0}}{s_{0}}\big{)}$. Let $T(0)=\varphi+r_0$. Then
\[
(Tf)(y)=\alpha f(Uy+x_0)+\langle\varphi,y\rangle+r_0,\;\;f\in{\rm conv}(X),\;y\in Y.
\]
\end{proof}
Recall that for a Banach space $X$, ${\rm conv}_{*}(X^*)$ denotes the cone of all extended real-valued $w^*$-l.s.c. convex functions defined on $X^*$.
\begin{theorem}\label{JFA}
Let $X$ be a Banach space. Then there is a fully order preserving mapping from ${\rm conv}(X)$ to ${\rm conv}_{*}(X^*)$ if and only if the following two conditions hold:
\begin{itemize}
\item[i)] $X$ is reflexive;
\item[ii)] there is a linear isomorphism from $X$ to $X^*$.
\end{itemize}
\end{theorem}
\begin{proof}
Sufficiency. Since $X$ is reflexive, we have ${\rm conv}_{*}(X^*)={\rm conv}(X^*).$ Let  $U: X\rightarrow X^*$ be an isomorphism. Then
\[
(Tf)(x^*)=f(U^{-1}x^*),\;f\in{\rm conv}(X),\;x^*\in X^*,
\]
defines a fully order preserving mapping $T: {\rm conv}(X)\rightarrow{\rm conv}(X^*) $.

Necessity. Suppose that $T: {\rm conv}(X)\rightarrow {\rm conv}_{*}(X^*)$  be a fully order preserving mapping. 
Note that ${\rm aff}_{*}(X^{*})$ (resp. ${\rm aff}(X)$) is perfect class of ${\rm conv}_{*}(X^{*})$ (resp. ${\rm conv}(X)$), where ${\rm aff}_{*}(X^{*})$ is the set of all $w^{*}$-continuous affine functions on $X^{*}$. Apply with Lemma \ref{sup-generating class}, we have ${\rm aff}(X)$ is a perfect class of ${\rm conv}(X)$. It now suffices to show that ${\rm aff}_{*}(X^{*})$ is a perfect class of ${\rm conv}_{*}(X^{*})$. Indeed, ${\rm aff}_{*}(X^{*})$ can be identified by $X\oplus \R$, which yields the purity of ${\rm aff}_{*}(X^{*})$. By the Hahn-Banach theorem, we have ${\rm aff}_{*}(X^{*})$ is a sup-generating class of ${\rm conv}_{*}(X^{*})$. It now suffices to show the $\varrho$-closeness of ${\rm aff}_{*}(X^{*})$. Indeed, assume that $\{f_{\alpha}\}_{\alpha\in\triangle}$ be a net of ${\rm aff}_{*}(X^{*})$ so that $f_{\alpha}=x_{\alpha}+r_{\alpha}\xrightarrow{\varrho} f\in {\rm C}_{{\rm conv}_{*}(X^{*})}$. Then, $f$ is affine, which implies that $f-f(0)$ is a $w^{*}$-continuous linear functional. Hence, $f-f(0)\in X$, that is $f=x+r$ for some $x\in X$ and $r\in \R$. Therefore, the perfectness of ${\rm aff}_{*}(X^{*})$ has been proven.

By an argument similar to the proof of Theorem \ref{thm}, we obtain that $T|_{{\rm aff}(X)}$ is a fully order preserving mapping, which is $\varrho$-to-$\varrho$ continuous (that is, pointwise-to-pointwise continuous) affine mapping from ${\rm aff}(X)$ to ${\rm aff}^*(X^*)$.  Let $S=T-T(0)$. Then
\[
S|_{{\rm aff}(X)}:X^*\oplus\mathbb R={\rm aff}(X)\rightarrow {\rm aff}^*(X^*)=X\oplus\mathbb R
\]
is a fully order preserving  $\varrho$-to-$\varrho$ continuous linear mapping. This entails that $X^*$ is isomorphic to $X$.
Note the $\varrho$-to-$\varrho$ continuity is equivalent to $w^*$-to-$w$ continuity in this setting. Then we obtain that $X$ is reflexive. Indeed, $T:X{^*}\to X$ is a linear isomorphism which is $w^{*}$-to-$w$ continuous and by the Alaoglu theorem, we have that $T(B_{X^{*}})$ is a weakly compact subset of $X$. By the open mapping theorem, it follows that there exists $\lambda>0$ so that $\lambda B_{X}\subseteq T(B_{X^{*}})$, which implies that $B_{X}$ is weakly compact, hence, yields the reflexivity of $X$.
\end{proof}
\begin{corollary}\label{reversing}
Suppose that $X$ is a Banach space. Then there is a fully order reversing mapping from ${\rm conv}(X)$ to itself if and only if $X$ is reflexive and isomorphic to its dual $X^*$.

\end{corollary}
\begin{proof}
By Theorem \ref{JFA}, it suffices to note that if $T: {\rm conv}(X)\rightarrow {\rm conv}(X)$ is fully order reversing if and only if $\mathcal FT: {\rm conv}(X)\rightarrow {\rm conv}_{*}(X^*)$
is fully order preserving, where $\mathcal F: {\rm conv}(X)\rightarrow {\rm conv}_{*}(X^*)$ denotes the Fenchel transform, which is fully order reversing.
\end{proof}
\begin{remark}
Clearly, for every Hilbert space $H$, there is a fully order reversing mapping of ${\rm conv}(H)$. It follows from Corollary \ref{reversing} that there are many non-Hilbert spaces satisfying the existence of fully order reversing mappings of the cone of l.s.c. convex functions defined on it. For example, let $X=L_p\times L_q$ with $1<p\neq 2<\infty$, $\frac{1}{p}+\frac{1}{q}=1$ and endows $X$ with the norm defined for $(x,y)\in X$ by
\begin{equation}\nonumber
\|(x,y)\|=\sqrt{\|x\|^2+\|y\|^2}.
\end{equation}
Then $X$ is reflexive and is isometric to $X^*$, which is not linearly isomorphic to a Hilbert space.
\end{remark}

\begin{theorem}\label{selfdual}
Suppose that $X$ is a Banach space. Then for every fully order reversing mapping $T$ from ${\rm conv}(X)$ onto itself there exist isomorphism $U:X\rightarrow X^*$, $x^*_0,\;\varphi\in X^*$, $\alpha>0$ and  $r_0\in\mathbb R$ so that
\[
(Tf)(x)=\alpha(\mathcal Ff)(Ux+x^*_0)+\langle\varphi,x\rangle+r_0,\;{\rm for\;all\;}x\in X,
\]
where $\mathcal F: {\rm conv}(X)\rightarrow {\rm conv}_{*}(X^*)$ is the Fenchel transform.
\end{theorem}
\begin{proof}
Suppose that $T: {\rm conv}(X)\rightarrow {\rm conv}(X)$ is a fully order reversing mapping. Then, $\mathcal{F}T$ is a fully order preserving mapping defined from ${\rm conv}(X)$ onto ${\rm conv}_{*}(X^{*})$. Applying Theorem \ref{presrving} and Theorem \ref{JFA} we have that there exist a $w^{*}$-to-$w$ linear isomorphism $V:X^{*}\to X$, $x_{0}$, $v_{0}\in X$, $\alpha>0$ and $\beta\in \R$ such that
\begin{equation}\label{representation for combination}
(\mathcal{F}T f)(y^{*})=\alpha f(Vy^{*}+x_{0})+\langle y^{*}, v_{0}\rangle+\beta,~y^{*}\in X^{*}.
\end{equation}
Since the Fenchel transform $\mathcal{F}$ satisfies that $\mathcal{F}\mathcal{F}={\rm id}$, then, by \eqref{representation for combination}, we have the following
\begin{equation}\label{representation for mapping T}
\begin{split}
(Tf)(x)&=(\mathcal{F}\mathcal{F}Tf)(x)=\sup\big{\{}\langle y^{*},x\rangle-\alpha f(Vy^{*}+x_{0})-\langle y^{*},v_{0}\rangle-\beta: y^{*}\in X^{*} \big{\}}\\
       &=\alpha \sup\big{\{}\langle y^{*},\frac{x-v_{0}}{\alpha}\rangle-f(Vy^{*}+x_{0}): y^{*}\in X^{*}\big{\}}-\beta\\
\end{split}
\end{equation}
Note that $V$ is a linear isomorphism defined from $X^{*}$ onto $X$, then, \eqref{representation for mapping T} becomes into
\[
\begin{split}
(Tf)(x)&=\alpha\sup\big{\{}\langle V^{-1}\xi,\frac{x-v_{0}}{\alpha}\rangle-f(\xi):\xi\in X\big{\}}-\langle V^{-1}{(x_{0})},x-v_{0}\rangle-\beta\\
       &=\alpha\sup\big{\{}\langle\frac{V^{-1*}(x-v_{0})}{\alpha},\xi\rangle-f(\xi): \xi\in X\big{\}}+\langle-V^{-1}x_{0},x\rangle+(\langle V^{-1}x_{0},v_{0}\rangle-\beta)\\
\end{split}
\]
Let $U=\frac{V^{-1*}}{\alpha}$, $x_{0}^{*}=-\frac{V^{-1*}(v_{0})}{\alpha}$, $\varphi=-V^{-1}(x_{0})$ and $r_{0}=\langle V^{-1}x_{0},v_{0}\rangle-\beta$, then we get that
\[
(Tf)(x)=\alpha(\mathcal Ff)(Ux+x^*_0)+\langle\varphi,x\rangle+r_0,\;{\rm for\;all\;}x\in X.
\]
\end{proof}

\section{An extended fundamental theorem of affine geometry}
In the following two sections, we shall discuss representation of fully order preserving mappings defined on the cone ${\rm semn}(X)$ of all extended real-valued l.s.c. seminorms on a Banach space $X$. Our main approach is to convert a fully order preserving mapping defined on  ${\rm semn}(X)$  to a lattice isomorphism defined on  $\mathscr{C}_{{\rm symm},*}(X)$ (the set of all nonempty $w^*$-closed symmetric convex subsets of $X^*$). Then by using the  generalized fundamental theorem of affine geometry  we obtain a linear isomorphism.

To begin with, we recall lattices of convex sets in the following.

By a lattice $\mathfrak S$, we mean that a partially ordered set such that any two of whose elements $a$ and $b$ have a least upper bound $a\vee b$ and a greatest lower bound $a\wedge b$.

\begin{example}
We order $\mathscr{C}_{*}(X^{*})$ (the collection of $w^*$-closed convex subsets of $X^*$)
by inclusion, and define two  operations $\wedge$ and $\vee$ on $\mathscr{C}_{*}(X^{*})$ as follows.
\begin{equation}\label{latticedef}
A\wedge B=A\cap B,\;\;A\vee B=\overline{{\rm co}}^{w^*}(A\cup B),\;\;{\rm for\;all}\; A, B\in\mathscr{C}_{*}(X^{*}),
\end{equation}
where $\overline{{\rm co}}^{w^*}(A)$ denotes the $w^*$-closed convex hull of $A$.
Then $(\mathscr{C}_{*}(X^{*}),\subseteq,\wedge,\vee)$ is a lattice.
\end{example}
\begin{example} Let $V$ be a  linear space over  $\mathbb R$. $\mathscr {P}(V)$ stands for the set of all finite dimensional linear subspaces on $V$. We now define two operations $\wedge$ and $\vee$ for $A$, $B\in\mathscr {P}(V)$ as follows,
\begin{equation}\nonumber
A\wedge B=A\cap B,\;\;\;\;A\vee B={\rm span}(A\cup B).
\end{equation}
Then $(\mathscr {P}(V),\subseteq,\wedge,\vee)$ is a lattice.
\end{example}

Recall that a mapping $\varphi$
from a lattice $\mathfrak {S}$ onto itself is a lattice isomorphism if $\varphi$ is bijective and satisfies
\begin{equation}\nonumber
\varphi (A\wedge B)=\varphi (A)\wedge\varphi (B),~~~\varphi (A\vee B)=\varphi (A)\vee\varphi (B),~~~\forall A,B\in \mathfrak {S}.
\end{equation}
The next lemma implies that the fully order preserving mapping $F$ defined by \eqref{composition1} is actually
a lattice isomorphism from $\mathscr{C}_{*}(X^{*})$ onto itself.

\begin{lemma}\label{latticepreserving}
Let $(\mathfrak {S},\subseteq,\wedge,\vee)$ be a lattice. Then $\varphi: \mathfrak {S}\rightarrow\mathfrak {S}$ is a fully order preserving
mapping  if and only if it is a lattice isomorphism.
\end{lemma}
\begin{proof} Sufficiency.  Suppose that $\varphi: \mathfrak {S}\rightarrow\mathfrak {S}$ is a lattice isomorphism. Let $A, B\in\mathfrak {S}$ with $A\leq B$. Note that $A\leq B\iff A\subseteq B\iff A\wedge B=A$. Then
\begin{equation}\nonumber
\varphi (A)=\varphi (A\wedge B)=\varphi (A)\wedge \varphi (B)\leq\varphi (B),
\end{equation}
i.e. $\varphi$ is order preserving. It is fully order preserving since it is bijective and since $\varphi^{-1}$ is also order preserving.

Necessity.  Suppose that $\varphi: \mathfrak {S}\rightarrow\mathfrak {S}$ is fully order preserving.
Given $A, B\in \mathfrak {S}$, we know $\varphi (A)\subseteq \varphi (A \vee B)$ and $\varphi (B)\subseteq \varphi (A \vee B)$, equivalently,
\[
\varphi (A)\vee\varphi (B)\leq\varphi (A \vee B).
\]
Since $\varphi$ is surjective, we can choose $D\in\mathfrak{S}$ so that $\varphi (A)\vee \varphi (B)=\varphi (D)$.
Therefore,  $A\subseteq D$ and $B\subseteq D$. Consequently, $A\vee B\subseteq D$, and further,
\[
\varphi (A \vee B)\subseteq \varphi (D)=\varphi (A)\vee \varphi (B).
\]
Hence,
\[
\varphi (A\vee B)=\varphi (A)\vee \varphi (B).
\]
We can show $\varphi (A\wedge B)=\varphi (A)\wedge\varphi (B)$ in the same way.
\end{proof}

Let $V$ be a  linear space over  $\mathbb R$ with $\dim(V)\geq 2$. $\mathscr {A}(V)$ stands for the set of all finite dimensional affine subspaces of $V$ and $\mathscr {P}(V)$  for  the set of all finite dimensional linear subspaces of $V$. Now, we define two lattice operations $\wedge$ and $\vee$ for  $A, B\in\mathscr {P}(V)$ as follows.
\begin{equation}\nonumber
A\wedge B=A\cap B,\;\;\;\;A\vee B={\rm span}(A\cup B).
\end{equation}
Then $(\mathscr {P}(V),\subseteq,\wedge,\vee)$ is a lattice.

The following two results (Lemma \ref{affine lem} and Theorem \ref{mainthm1} are called the fundamental theorem of affine geometry when the linear space $V$ is finite dimensional (see, for instance, \cite{Ar} and \cite[pp.56-57]{G-W}). The fundamental theorem of affine geometry also plays an important role in projective geometry.
\begin{lemma}\label{affine lem}
Suppose that $V$ is a linear space with $\dim(V)\geq2$,  and that $\alpha: V\rightarrow V$ is a mapping. If $\alpha$ fully order preserving from $\mathscr {A}(V)$ onto itself with $\alpha (0)=0$, then $\alpha$ is linear bijection.
\end{lemma}
\begin{proof}
Given any  $u,v\in V$, let
\begin{equation}\nonumber
L={\rm span}\{u,v\}={\rm aff}\{u,v,0\}.
\end{equation}
Since $\alpha$ is  order preserving,  it maps every one dimensional subspace of $L$ into a one dimensional subspace of $M\coloneqq\alpha(L)$. Therefore,
\[
\alpha (L)={\rm aff}\{\alpha(u),\alpha(v),0\}=M.
\]
Consequently, $\alpha$ (restricted to $L$) is a fully order preserving from $\mathscr {A}(L)$ onto $\mathscr {A}(M)$. Applying the corresponding classical fundamental theorem, we obtain that for all $r\in \mathbb {R}$
\begin{equation}\nonumber
\alpha (u+v)=\alpha (u)+\alpha (v),\;\;\alpha (ru)=r\alpha (u),\;\;\alpha (rv)=r\alpha (v).
\end{equation}
Since $u, v$ are arbitrary, $\alpha$ is a bijective linear map of $V$ onto itself.
\end{proof}

\begin{theorem}\label{mainthm1}
Suppose that $V$ is a linear space with $\dim(V)\geq 3$.
Let $\pi:\mathscr P(V)\rightarrow\mathscr P(V) $ be a fully order preserving mapping. Then there exists
a bijective linear mapping $\lambda:V\rightarrow V$  so that
\begin{equation}\nonumber
\pi (M)=\lambda (M),\;\;{\rm for\;all\;}M\in\mathscr{P}(V).
\end{equation}
\end{theorem}

\begin{proof}
Let $H$ be a hyperplane of $V$ and  $a\in V\setminus H$. Then $V=H\oplus\mathbb Ra$. Let $\mathscr{P}(H)$ be the set
of all finite dimensional subspaces of $H$, and let $\mathscr {F}=\{\pi (M):M \in \mathscr {P}(H)\}$.
Then we claim that $H^{\prime}\coloneqq\bigcup\{N\in\mathscr {F}\}$ is a hyperplane of $V$. Indeed, since $\pi:\mathscr P(V)\rightarrow\mathscr P(V) $ is fully order preserving, it is easy to observe that $H^{\prime}$ is a linear subspace of $V$. Choose any $u\in V\setminus H$. Then $\pi(\mathbb Ru)\nsubseteq H^{\prime}$, which entails that $H^{\prime}$ is a proper subspace of $V$. Suppose that $H^{\prime}$ is not a hyperplane. Then there is a subspace $N$ of $V$ with $\dim N\geq2$  so that $H^{\prime}+N=H^{\prime}\oplus N=V.$ Choose any two dimensional subspace $M$ of $N$. Since $\pi^{-1}$ is also fully order preserving,  $\pi^{-1}(M)$ is also a two dimensional subspace of $V$ with $H\cap\pi^{-1}(M)=\{0\}$. This is a contradiction.

Now, let $a^\prime\in V$ be so that $\pi(\mathbb Ra)=\mathbb Ra^\prime.$ Denote by $\mathscr {A}(a+H)$ (resp. $\mathscr {A}(a^\prime+H^\prime)$), the set of all finite dimensional affine subspaces of $a+H$ (resp. $a^\prime+H^\prime$). We define a mapping $\alpha: \mathscr {A}(a+H)\rightarrow\mathscr {A}(a^\prime+H^\prime)$ by
\[
\alpha (S)={\rm span}(S)\cap (a^{\prime}+H^{\prime}),\;\; S\in \mathscr {A}(a+H).
\]
Since $\pi$ is fully order preserving,  $\alpha$ must be fully order preserving.
Note that for each $S\in\mathscr {A}(a+H)$, there is $L \in \mathscr {P}(H)$ so that $S=a+L$.
Then
\[
\tilde{\alpha}(L)\coloneqq\alpha(a+L)-a^{\prime},\;\;L\in\mathscr{A}(H)
\]
defines a fully order preserving mapping $\tilde{\alpha}:\mathscr{A}(H)\rightarrow\mathscr {A}(H^{\prime})$ satisfying $\tilde{\alpha}(0)=0$.
Applying Lemma \ref{affine lem}, we obtain that $\tilde{\alpha}:  H\rightarrow H^{\prime}$ is linear bijective.

Finally, we define $\lambda:V\rightarrow V$  by
\[
\lambda(ra+h)\coloneqq ra^{\prime}+\tilde{\alpha}(h),\;\;r\in\mathbb{R},\;h\in H.
\]
Then the bijective linear mapping $\lambda$ satisfies
\[
\pi (M)=\lambda (M),\;\;{\rm for\;all\;}M\in\mathscr{P}(V).
\]
\end{proof}

\begin{remark}
In finite dimensional spaces, the fundamental theorem of affine geometry has been generalized in various ways in \cite{A-A-S16} and \cite{S-V}, which contain valuable  references and historical remarks of the fundamental theorem. The classical fundamental theorem has been applied to study representation of lattice isomorphisms on set lattices consisting of convex sets in finite dimensional spaces. For detailed information, we refer the reader to \cite{Gr, SchR, Sl}.
\end{remark}

\section{Fully order preserving mappings on the cone of seminorms}
This section is devoted to the study of characterizations of fully order preserving mappings defined on cone of l.s.c. seminorms.

We denote by  ${\rm mink}(X)$ (resp. ${\rm subl}(X)$, ${\rm semn}(X)$) by the set of all l.s.c. Minkowski  (resp. sublinear functions, seminorms) defined on $X$.  ${\mathscr C}_{0,*}(X^*)$ stands for the set  of all $w^*$-closed convex subsets of $X^*$ containing the origin,  ${\mathscr C}_{*}(X^*)$ for the set of all nonempty $w^*$-closed convex subsets, and  $\mathscr C_{{\rm symm,}*}(X^*)$ for the set of all $w^*$-closed symmetric convex subsets of $X^*$.

If we order ${\mathscr C}_{0,*}(X^*)$ and  $\mathscr C_{{\rm symm,}*}(X^*)$ by inclusion of sets,
and define two  operations $\wedge$ and $\vee$ on ${\mathscr C}_{0,*}(X^*)$ and  $\mathscr C_{{\rm symm,}*}(X^*)$ as follows:
\begin{equation}\label{latticedef}
A\wedge B=A\cap B,\;\;A\vee B=\overline{{\rm co}}^{w^*}(A\cup B),\;\;{\rm for\;all}\; A, B,
\end{equation}
where $\overline{{\rm co}}^{w^*}(A)$ denote the $w^*$-closed convex hull of $A$,
then ${\mathscr C}_{0,*}(X^*)\coloneqq({\mathscr C}_{0,*}(X^*),\subseteq,\wedge,\vee)$ and $\mathscr C_{{\rm symm,}*}(X^*)\coloneqq({\mathscr C}_{{\rm symm,}*}(X^*),\subseteq,\wedge,\vee)$ are lattices.

For a convex function $f$ defined on $X$, let
\begin{equation}\label{subdifferential}
\mathcal D(f)=\partial f(X),
\end{equation}
the image of $f$ under the subdifferential operator $\partial$, i.e. $\mathcal D(f)=\bigcup\limits_{x\in X}\partial f(x)$;
and let
\begin{equation}\label{supportfunction}
\mathcal S(A)=\sup_{x^*\in A}\langle x^*,\cdot\rangle\coloneqq\sigma_A.
\end{equation}
It is clear that for a sublinear function $f$, then $\mathcal{D}(f)=\partial f(0)$ with $f=\sigma_{\partial f(0)}$. And the following propositions are immediately from the definitions of subdifferential operator and support function.
\begin{proposition}\label{transformation}
With the notions as above, we have the following properties.
\begin{itemize}
\item[i)] $\mathcal D: {\rm subl}(X)\rightarrow {\mathscr C}_{*}(X^*)$ (resp. ${\rm mink}(X)\rightarrow {\mathscr C}_{0,*}(X^*)$ or ${\rm semn}(X)\rightarrow {\mathscr C}_{{\rm symm},*}(X^*)$) is a fully order preserving mapping;
\item[ii)] $\mathcal S: {\mathscr C}_{*}(X^*)\rightarrow{\rm subl}(X) $ (resp. ${\mathscr C}_{0,*}(X^*)\rightarrow {\rm mink}(X)$ or ${\mathscr C}_{{\rm symm},*}(X^*)\rightarrow{\rm semn}(X) $) is a fully order preserving mapping.
\end{itemize}
\end{proposition}
\begin{remark}
By the Proposition \ref{transformation} and the fact that combination of fully order preserving mappings is again fully order preserving, then the following statements are immediately.
\begin{itemize}
\item[i)] For any fully order preserving mapping ${\rm subl}(X)\rightarrow {\rm subl}(X)$ (resp. $T:{\rm mink}(X)\rightarrow {\rm mink}(X)$, or, ${\rm semn}(X)\rightarrow {\rm semn}(X)$), $F_{T}\coloneqq\mathcal DT\mathcal S$ is a fully order preserving mapping, hence, a lattice isomorphism from $\mathscr{C}_{*}(X^*)$ (resp. $\mathscr{C}_{0,*}(X^{*})$ or $\mathscr{C}_{{\rm symm},*}(X^*)$) onto itself.
\item[ii)] For any fully order preserving mapping $T:{\mathscr C}_{0,*}(X^*)\rightarrow {\mathscr C}_{0,*}(X^*)$ (resp. ${\mathscr C}_{*}(X^*)\rightarrow {\mathscr C}_{*}(X^*)$ or ${\mathscr C}_{{\rm symm},*}(X^*)\rightarrow {\mathscr C}_{{\rm symm},*}(X^*)$), $F_{T}\coloneqq\mathcal{S} T\mathcal{D}$ is a fully order preserving mapping defined from ${\rm subl}(X)$ (resp. ${\rm mink}(X)$ or ${\rm semn}(X)$).
\end{itemize}
\end{remark}

For the dual $X^*$ of a Banach space $X$, we denote
\[
|X^*|\coloneqq\{|x^*|\coloneqq x^*\vee-x^*: x^*\in X^*\},
\]
and
\[
[X^*]\coloneqq\{[-x^*,x^*]: x^*\in X^*\},
\]
where $[-x^{*},x^{*}]=\{\lambda x^{*}:\lambda\in[-1,1]\}$.

\begin{lemma}\label{semibasic}
Let $F: {\mathscr C}_{{\rm symm},*}(X^*)\rightarrow {\mathscr C}_{{\rm symm},*}(X^*)$ be a fully order preserving mapping. Then
\begin{itemize}
\item[i)] its restriction to $[X^*]$, $F|_{[X^*]}$ is also fully order preserving from $[X^*]$ onto itself;
\item[ii)] there exists a bijective linear mapping $\Lambda:X^*\rightarrow X^*$  so that
\begin{equation}\nonumber
F(M)=\Lambda (M),\;\;{\rm for\;all\;}M\in\mathscr{P}(X^*);
\end{equation}
\item[iii)] moreover, we can claim
\begin{equation}\nonumber
F([-x^*,x^*])=\Lambda ([-x^*,x^*]),\;\;{\rm for\;all\;}x^*\in X^*.
\end{equation}
\end{itemize}
\end{lemma}
\begin{proof}
{\rm i)}. Let $F_{\mathcal{S}}=\mathcal{S}F$, where $\mathcal{S}$ is defined as in \eqref{supportfunction}. Then
\[
F_{\mathcal{S}}: \mathscr{C}_{{\rm symm},*}(X^{*})\rightarrow {\rm semn}(X)
\]
is a fully order preserving mapping. Since  for any fixed $x^*\in X^*$,
\[
A\coloneqq\{[-z^*,z^*]:z^{*}\in[-x^{*},x^{*}]\}=\{[-\lambda x^*, \lambda x^*]: 0\leq\lambda\leq1\}
\]
is an ordered set in $[X^*]$, and
\[
F_{\mathcal{S}}(A)=\{F_{\mathcal{S}}\big{(}[-z^*,z^*]\big{)}: z^{*}\in [-x^{*},x^{*}]\}=\{p\in{\rm semn}(X): p\leq F_{\mathcal S}([-x^*,x^*])\}
\]
is also an ordered set in ${\rm semn}(X)$, which yields that $F_{\mathcal S}([-x^*,x^*])=|y^*|$ for some $y^*\in X^*$ satisfying $y^*=0$ if and only if $x^*=0$. Hence,
\begin{equation}\label{interval to interval}
F\big{(}[-x^*,x^*]\big{)}=[-y^*,y^*].
\end{equation}
By noting that $\mathbb{R} x^{*}=\bigcup\limits_{r>0}[-r x^{*},r x^{*}]$ for all $x^{*}\in X^{*}$ and by \eqref{interval to interval}, we have that
\[
F(\mathbb Rx^*)=\mathbb Ry^*\coloneqq\{ry^*: r\in\mathbb R\},
\]
i.e. $F$ maps each one dimensional subspace  onto a one dimensional subspace of $X^*$. Since $F: {\mathscr C}_{{\rm symm},*}(X^*)\rightarrow {\mathscr C}_{{\rm symm},*}(X^*)$ is fully order preserving, $F|_{[X^*]}$ is also fully order preserving from $[X^*]$ onto itself.

{\rm ii)}. It suffices to prove that $F$ must map finite dimensional subspaces to finite dimensional subspaces with the same dimension. It is clear that $F$ must map one dimensional subspaces to one dimensional subspaces. By induction,
for any $1\leq k\leq n$, we have so that $\text{dim}(V)=k$ if and only if $\text{dim}(F(V))=k$. Suppose that $V$ is a
subspace of $X^*$ so that $\text{dim}(V)=n+1$, then there exists $\{x_i^*\}_{i=1}^{n+1}$ such that $V=\bigvee\limits_{i=1}^{n+1}\mathbb{R}{x_i^*}$. Since $F$ is a fully order preserving mapping of $\mathscr{C}_{symm,*}(X^*)$, by Lemma \ref{latticepreserving}, $F$ is a lattice isomorphism of $\mathscr {C}_{symm,*}(X^*)$. Thus, $F(V)=\bigvee\limits_{i=1}^{n+1}F(\mathbb{R}{x_i^*})$. It follows that $\text{dim}(F(V))\leq n+1$. If $\text{dim}(F(V))<n+1$ contradicts to the assumption. Note that $F^{-1}$ is also a fully order preserving mapping of $\mathscr{C}_{symm,*}(X^*)$, then $\text{dim}(V)=n+1$ if and only if $\text{dim}(F(V))=n+1$.
Hence, $\text{dim}(V)=\text{dim}(F(V))$ for every finite dimensional subspace $V\subseteq X^*$.
Therefore, by Theorem \ref{mainthm1},

\begin{equation}\label{linearbijection}
F(M)=\Lambda (M),\;\;{\rm for\;all\;}M\in\mathscr{P}(X^*).
\end{equation}
Therefore, we have shown {\rm ii)}.

{\rm iii)}. Fix any $x_0^*\in X^*\setminus\{0\}$. By \eqref{linearbijection},
\[
F(\mathbb Rx^*_0)=\Lambda(\mathbb Rx^*_0).
\]
Since $F[-x^*_0,x^*_0]\subseteq\Lambda(\mathbb Rx^*_0),$ and since $\Lambda$ is linear, there exists $\alpha>0$ so that $F[-x^*_0,x^*_0]=\alpha\Lambda([-x^*_0,x^*_0])$. Without loss of generality, we can assume $\alpha=1$.
Fix any $x^*\in X^*\setminus\{0\}.$

Case \uppercase\expandafter{\romannumeral1}. If $x^*$ is linearly independent of $x^*_0$, then
\[
[-x^*,x^*]=\Big(\mathbb {R}x^*\Big)\bigcap\overline{\rm co}^{w^*}\Big(\mathbb {R}(x^*-x_0^*)\cup[-x_0^*,x_0^*]\Big).
\]
Since $F$ is a lattice isomorphism on ${\mathscr C}_{{\rm symm},*}(X^*)$ and $F$ is continuous when restricted to
finite dimensional subspaces. Hence,
\begin{eqnarray}\nonumber
F\big{(}[-x^*,x^*]\big{)}&=&F\Big((\mathbb {R}x^*)\bigcap\overline{\rm co}^{w^*}\big(\mathbb {R}(x^*-x_0^*)\cup[-x_0^*,x_0^*]\big)\Big)\\\nonumber
&=&F(\mathbb {R}x^*)\bigcap\overline{\rm co}^{w^*}\Big(F\big(\mathbb {R}(x^*-x_0^*)\big)\cup F\big([-x_0^*,x_0^*]\big)\Big)\\\nonumber
&=&F(\mathbb {R}x^*)\bigcap\overline{\rm co}^{w^*}\Big(F\big(\mathbb {R}(x^*-x_0^*)\big)\cup F[-x_0^*,x_0^*]\Big)\\\nonumber
&=&\Lambda(\mathbb {R}x^*)\bigcap\overline{\rm co}^{w^*}\Big(\Lambda\big(\mathbb {R}(x^*-x_0^*)\big)\cup\Lambda [-x_0^*,x_0^*]\Big)\\\label{e}
&=&\Lambda\Big(\mathbb {R}x^*\cap\overline{\rm co}^{w^*}\big(\mathbb {R}(x^*-x_0^*)\cup [-x_0^*,x_0^*]\big)\Big)\\\nonumber
&=&\Lambda\big{(}[-x^*,x^*]\big{)},\\\nonumber
\end{eqnarray}
i.e.
\[
F\big{(}[-x^*,x^*]\big{)}=\Lambda\big{(}[-x^*,x^*]\big{)}.
\]
Case \uppercase\expandafter{\romannumeral2}. If $x^*=\beta x^*_0$ for some $\beta\not=0$, by choosing any $y_0^*$, which is linearly independent of $x_0^*$, then applying Case \uppercase\expandafter{\romannumeral1}, we obtain
\[
F\big{(}[-y_0^*,,y_0^*]\big{)}=\Lambda\big{(}[-y_0^*,y_0^*]\big{)}.
\]
Now, we repeat the previous procedure of the proof of {\rm iii)}, but substitute $y^*_0$ for $x^*_0$. Since $x^*$ is linearly independent of $y^*_0$, we have
\[
F\big{(}[-x^*,,x^*]\big{)}=\Lambda\big{(}[-x^*,x^*]\big{)}.
\]

Hence, {\rm iii)} is proven. Consequently, the proof is completed.
\end{proof}

\begin{lemma}\label{continuous-semi}
Suppose $T:{\rm semn}(X)\rightarrow{\rm semn}(X)$ is  fully order preserving. Then
\begin{itemize}
\item[i)] $T$ maps every continuous seminorm $p$ into a continuous seminorm $T(p)$;
\item[ii)] $T|_{{\rm C}_{{\rm semn}}(X)}$ is again fully order preserving from ${\rm C}_{{\rm semn}}(X)$ onto itself,
\end{itemize}
where ${\rm C}_{{\rm semn}}(X)$ denotes the cone of all continuous seminorms on $X$.
\end{lemma}
\begin{proof}
{\rm i)}. For each $p\in{\rm semn}(X)$ is l.s.c., then it is continuous if and only if $p(x)<\infty$ for all $x\in X$. Let $p\in{\rm C}_{{\rm semn}}(X)$. Suppose, to the contrary, that there  is $x_0\in X$ so that $(Tp)(x_0)=\infty$. Then $(Tp)(rx_0)=|r|(Tp)(x_0)=\infty$ for every $0\neq r\in \mathbb{R}$.
Let
\begin{equation}\nonumber
q(x)=\delta_{\mathbb{R}x_0}(x)=
\begin{cases}
0, &x\in\mathbb{R}x_0,\\
\infty, &\text{otherwise}.
\end{cases}
\end{equation}
Then $q\in {\rm semn}(X)$. Therefore,  there is $g\in{\rm semn}(X)$ so that $Tg=q$.
Hence,
\begin{equation}\nonumber
q\vee Tp=Tg\vee Tp=T(p\vee g)=\delta_{0}.
\end{equation}
Since $\delta_{0}$ is the maximum of ${\rm semn}(X)$, and since $T$ is  fully order preserving, we know $\delta_{0}$ is a fixed point of $T$, i.e. $p\vee g=T(p\vee g)=\delta_{0}$.
Since $f(x)<\infty,\forall x\in X$,  $g=\delta_{0}$, and this contradicts to
$Tg=q=\delta_{\mathbb{R}x_0}(\not=\delta_{0})$.

{\rm ii)}. By {\rm i)} we have just proven, we see that $T|_{{\rm C}_{{\rm semn}}(X)}$ is  order preserving from ${\rm C}_{{\rm semn}}(X)$ into itself. Note that $T^{-1}: {\rm semn}(X)\rightarrow{\rm semn}(X)$ is  also fully order preserving. Then by {\rm i)} again, we get $T^{-1}|_{{\rm C}_{{\rm semn}}(X)}$ is  order preserving from ${\rm C}_{{\rm semn}}(X)$ into itself. Thus, $T|_{{\rm C}_{{\rm semn}}(X)}$ is again fully order preserving.
\end{proof}
\begin{theorem}\label{representation-semi}
Suppose that $T:{\rm semn}(X)\rightarrow {\rm semn}(X)$ is  fully order preserving. Then there exists a unique isomorphism
$U: X\rightarrow X$ so that
\[
Tf(x)=f(Ex), {\rm for\;all}\;f\in{\rm semn}(X), \;  x\in X;
\]
where $E\in\{\pm U\}$.
\end{theorem}

\begin{proof}
Let $\mathcal D$ (resp. $\mathcal S$) be defined as \eqref{subdifferential} (resp. \eqref{supportfunction}). Then by Proposition \ref{transformation}, we have that both $\mathcal{D}$ and $\mathcal{S}$ are fully order preserving mappings.
Since the composition of fully order preserving mappings is again fully order preserving, then
\begin{equation}\label{induce}
F\coloneqq \mathcal DT\mathcal S:{\mathscr C}_{{\rm symm},*}(X^*)\rightarrow{\mathscr C}_{{\rm symm},*}(X^*)
\end{equation}
is fully order preserving.
Applying Lemma \ref{semibasic} {\rm ii)}, there is a bijective linear mapping $\Lambda: X^*\rightarrow X^*$ so that
\[
F(M)=\Lambda (M),\;\;{\rm for\;all\;finite\;dimensional \;subspace\;}M\subseteq X^*
\]
and
\begin{equation}\label{symmetricinterval}
F([-x^*,x^*])=\Lambda ([-x^*,x^*]),\;\;{\rm for\;all\;}x^*\in X^*.
\end{equation}
Note here that for every $w^{*}$-compact symmetric convex subset $A$, we have that
\[
A=\bigcup\limits_{x^{*}\in A}[-x^{*},x^{*}].
\]
We now claim that
\begin{equation}\label{representation}
F(A)=\bigcup\limits_{x^{*}\in A}F\big{(}[-x^{*},x^{*}] \big{)}
\end{equation}
for all $A\in\mathscr{C}_{{\rm symm},*}(X^{*})$. Indeed, by the fact that $F$ is order preserving, we have
\[
\bigcup\limits_{x^{*}\in A}F\big{(}[-x^{*},x^{*}]\big{)}\subseteq F(A).
\]
Conversely, for any $y^{*}\in F(A)$, by the fact that $F^{-1}$ is also order preserving, then it entails that
\[
F^{-1}\big{(}[-y^{*},y^{*}]\big{)}\subseteq A.
\]
Due to Lemma \ref{semibasic} {\rm i)}, it follows that there exists $z^{*}\in A$ with $F^{-1}\big{(}[-y^{*},y^{*}]\big{)}=[-z^{*},z^{*}]$, i.e. $F\big{(}[-z^{*},z^{*}]\big{)}=[-y^{*},y^{*}]$. Note that $y^{*}$ is arbitrary in $F(A)$, which yields that
\[
F(A)\subseteq \bigcup\limits_{x^{*}\in A} F\big{(}[-x^{*},x^{*}]\big{)}.
\]
Combining with \eqref{symmetricinterval} and \eqref{representation} we have that
\[
F(A)=\Lambda(A),~A\in {\mathscr C}_{{\rm symm},*}(X^*).
\]
Applying Lemma \ref{continuous-semi} {\rm ii)}, $T$ maps each continuous seminorm into a continuous one and note that $F$ is a fully order preserving mapping induced by $T$. Hence, for every $w^{*}$-compact symmetric convex subset $A$,
\[
\Lambda(A)=F(A)=\partial(T\sigma_{A})(0),
\]
which entails that $\Lambda$ maps $w^{*}$-compact symmetric convex subsets of $X^{*}$ to $w^{*}$-compact symmetric convex subsets. By the Banach-Steinhauss theorem, it yields that $\Lambda$ is continuous in norm.
Furthermore, we shall show that $\Lambda$ is $w^{*}$-to-$w^{*}$ continuous. Indeed, thanks to the Grothendieck's dual characterization of completeness (see \cite[p.149 Corollary 2]{SchH}), it suffices to show that if $\{x^{*}_{\alpha}\}_{\alpha}$ is a bounded net in $X^{*}$ with $x^{*}_{\alpha}\xrightarrow{w^{*}} 0$, then
$\Lambda(x^{*}_{\alpha})\xrightarrow{w^{*}}0$. Note equality \eqref{symmetricinterval} above is equivalent to
\begin{equation}\label{seminorm and continuity}
T(|x^*|)=|\Lambda (x^*)|,\;\;{\rm for\;all\;}x^*\in X^*.
\end{equation}
Then, apply with Theorem \ref{continuity theorem} and \eqref{seminorm and continuity}, it follows that
\[
\lim\limits_{\alpha}|\Lambda(x^{*}_{\alpha})(x)|=\lim\limits_{\alpha}\big{(}T|x^{*}_{\alpha}|\big{)}(x)=0,
\]
for all $x\in X$, which yields that $\Lambda$ is $w^{*}$-to-$w^{*}$ continuous. Consequently, there exists a linear isomorphism $U:X\to X$ so that $\Lambda=U^*$.

Since
\[
f(x)=\sup\{|\langle x^*,x\rangle|:x^*\in\partial f(0)\}
\]
for all  $f\in{\rm semn}(X)$ and $x\in X$,
\begin{eqnarray}\nonumber
({T}f)(x)&=&\sup\{{T} |\langle x^*,x\rangle|: x^*\in \partial f(0)\}\\\nonumber
&=&\sup\{|\Lambda(x^*)|(x):x^*\in \partial f(0)\}\\\nonumber
&=&\sup\{|x^*|(Ux):x^*\in \partial f(0)\}\\\nonumber
&=&f(Ux).
\end{eqnarray}
We have proven
\[
{T}(f)(x)=f(Ux),\;\;\forall\;f\in{\rm semn}(X),\; x\in X.
\]
It remains to show that $U$ is a unique. Suppose that there is an isomorphism  $E: X\rightarrow X$ so that
\[
{T}(f)(x)=f(Ex),\;\;\forall\;f\in{\rm semn}(X),\; x\in X.
\]
Then
\[
f(Ex)=f(Ux),\;\;\forall\;f\in{\rm semn}(X),\; x\in X.
\]
In particular (by taking $f=|x^*|$), we obtain
\[
|\langle x^*, Ex\rangle|=|\langle x^*, Ux\rangle|,\;\;\forall\;x^*\in X^*,\; x\in X.
\]
This, in turn, implies $E=\pm U$.
\end{proof}

\section{Fully order preserving mappings on the cone of Minkowski functionals}
We will show in this section that characterizations of fully order preserving mappings defined on certain classes of convex functions (such as Minkowski functionals, continuous norms, positive homogeneous convex functions of degree $p$ et. al) could be deduced from characterizations of fully order preserving mappings defined on seminorms.

Recall that a Minkowski functional  is a non-negative positively homogeneous convex function. For a Banach space $X$, we denote by   ${\rm mink}(X)$  the cone of all extended real-valued l.s.c. Minkowski functionals defined on $X$, and  ${\mathscr C}_{0,*}(X^*)$ is defined as in Section 8,  the cone of all $w^*$-closed convex subsets of $X^*$ containing the origin.

In this section, we shall give a representation theorem for fully order preserving mappings on the cone  ${\rm mink}(X)$. Denote that $[X^*]_{0} \coloneqq\{[0,x^*]: x^*\in X^*\}$.

\begin{theorem}\label{representation-mink}
Let ${T}: {\rm mink}(X)\to {\rm mink}(X)$ be  fully order preserving. Then
\begin{itemize}
\item[i)] $T|_{{\rm semn}(X)}$ (the restriction of $T$ to  ${\rm semn}(X)$) is again fully order preserving;
\item[ii)] there is a unique isomorphism $E: X\rightarrow X$ so that
\begin{equation}\label{representationmink}
 {T}f(x)=f(Ex),\;\;\forall f\in {\rm mink}(X), x\in X.
\end{equation}
\end{itemize}
\end{theorem}
\begin{proof}
{\rm i)}. Let $\mathcal D: {\rm mink}(X)\rightarrow{\mathscr C}_{0,*}(X^*)$ (resp. $\mathcal S: {\mathscr C}_{0,*}(X^*)\rightarrow{\rm mink}(X)$ ) be defined as \eqref{subdifferential} (resp. \eqref{supportfunction}), i.e.
\[
\mathcal D(f)=\partial f(0), \;\;f\in{\rm mink}(X).
\]
\[
({\rm resp.}\; \mathcal S(C)(x)=\sigma_{C}(x)=\sup_{x^*\in C}\langle x^*,x\rangle, \;\;C\in{\mathscr C}_{0,*}(X^*),\;x\in X.)
\]
Then
$F=\mathcal D{T}\mathcal S:{\mathscr C}_{0,*}(X^*)\rightarrow{\mathscr C}_{0,*}(X^*) $  is a fully order preserving mapping with $F(\{0\})=\{0\}$, the minimal element of the lattice $({\mathscr C}_{0,*}(X^*),\subseteq,\wedge,\vee)$.
By an argument similar to the proof of Lemma \ref{semibasic} i), we get that $F$ maps $[X^*]_0$ onto itself. Next, we show that
\begin{equation}\label{mink-semi}
F([0,-x^*]\subseteq \R{F}[0,x^*],\;\;{\rm for\;all}\;x^*\in X^*.
\end{equation}

Suppose, to the contrary, that there exists $x^*\in X^*$ so that if ${F}[0,-x^*]\not\subseteq \R{F}[0,x^*]$.
Then
\begin{equation}\nonumber
C\coloneqq{\rm co}\Big(F([0,-x^*])\bigcup F([0,x^*])\Big)=F([0,-x^*])\vee F([0,x^*])
\end{equation}
is a convex set but not a segment. Consequently,  there is $[0,y_0^*]\subseteq{C}$ satisfying
\[
[0,y_0^*]\not\subseteq{F}([0,-x^*])\bigcup{F}([0,x^*]).
\]
Since $F:[X^*]_0\rightarrow [X^*]_0$ is bijective, there exists $[0,z_0^*]\in[X^*]_0$ so that $F([0,z_0^*])=[0,y_0^*]$. Since $[0,y_0^*]\subseteq{C}$, we have that
\[
[0,z_0^*]\subseteq [0,-x^*]\vee [0,x^*]=[-x^*,x^*].
\]
It follows that
\[
[0,y^*_0]=F([0,z_0^*])\subseteq F([0,-x^*])\bigcup F([0,x^*]),
\]
which is a contradiction.
Hence, we have shown that $F$ maps lines into lines. By induction, it follows that $F$ induces a fully order preserving mapping defined from
$\mathscr{P}(X^*)$ onto itself. Applying Theorem \ref{mainthm1} and the proofs in Lemma \ref{semibasic} and Theorem \ref{representation-semi} we have
there exists a unique isomorphism $U: X\rightarrow X$ so that
\begin{equation}\label{semi-mink}
F\big{(}[-x^{*},x^{*}]\big{)}=U^{*}\big{(}[-x^{*},x^{*}]\big{)},\;\;{\rm for\;all\;}x^*\in X^*.
\end{equation}
Since $Tf=\sigma_{F(\partial f(0))}$ for all $f\in {\rm mink}(X)$ and $[X^{*}]$ is a sup-generating class of ${\rm semn}(X)$, then, by \eqref{semi-mink}, we get that $T|_{{\rm semn}(X)}$ is a fully order preserving mapping defined from ${\rm semn}(X)$ onto itself. Hence, by Theorem \ref{representation-semi}, we have that
\begin{equation}\label{semi-mink restriction}
{T|_{{\rm semn}(X)}}f(x)=f(Ex),~~~\forall f\in{\rm semn}(X),  x\in X,
\end{equation}
where $E\in\{\pm U\}$.

{\rm ii)}. By \eqref{semi-mink} and the fact that $[X^{*}]_{0}=\{[0,x^{*}]:x^{*}\in X^{*}\}$ is a sup-generating class of ${\rm mink}(X)$, it follows that there are only the following two possible cases for $T$:
\begin{equation}\label{representation-mink1}
{T}f(x)=f(Ux),~~~\forall f\in{\rm mink}(X),~ x\in X,
\end{equation}
or
\begin{equation}\label{representation-mink2}
{T}f(x)=f(-Ux),~~~\forall f\in{\rm mink}(X),~ x\in X.
\end{equation}

It follows from \eqref{mink-semi} and \eqref{semi-mink} that either
\[
F\big{(}[0,x^*] \big{)}=U^{*} \big{(}[0,x^*]\big{)},\;\;{\rm for\;all\;}x^*\in X^*,
\]
or
\[
F\big{(}[0,x^*]\big{)}=-U^{*}\big{(}[0,x^*]\big{)},\;\;{\rm for\;all\;}x^*\in X^*.
\]

Clearly, the former case is equivalent to \eqref{representation-mink1}, i.e. $E=U$;  and the later case is equivalent to \eqref{representation-mink2}, i.e. $E=-U$.
\end{proof}

We use ${\rm norm}(X)$ to denote all (equivalent) norms on $X$ and  ${\mathscr C}_{{\rm ssymm,c},*}(X^*) $ the set of all solid (i.e. with nonempty interior) $w^*$-compact symmetric convex subsets of $X^*$.
\begin{theorem}
Suppose that $X$ is a Banach space. Then
\begin{itemize}
\item[i)] every fully order preserving mapping $T:{\rm norm}(X)\rightarrow {\rm norm}(X)$ is a restriction of a fully order preserving mapping $S:{\rm semn}(X)\rightarrow{\rm semn}(X)$;
\item[ii)] every fully order preserving mapping $T:{\rm semn}(X)\rightarrow {\rm semn}(X)$ is  an extension of a fully order preserving mapping $S:{\rm norm}(X)\rightarrow{\rm norm}(X)$.
\end{itemize}
\end{theorem}
\begin{proof}
{\rm i)}. Suppose $T:{\rm norm}(X)\rightarrow {\rm norm}(X)$ is a fully order preserving mapping. Let $\mathcal D$ and $\mathcal S$ be defined as in previous sections, i.e.
$\mathcal T=\mathcal DT\mathcal S$,
then $\mathcal T$ is a fully order preserving mapping from ${\mathscr C}_{{\rm ssymm,c},*}(X^*)$ onto itself.
Let us denote by ${\mathscr C}_{{\rm symm,c},*}(X^*)$ the set of all symmetric $w^*$-compact convex subset of $X^*$.
We now extend $\mathcal T$ to be a fully order preserving mapping from ${\mathscr C}_{{\rm symm,c},*}(X^*)$ onto itself.
\newline
More precisely,
given any $A\in {\mathscr C}_{{\rm symm,c},*}(X^*)$, let $A_q=A\vee (q B_{X^*})=\overline{\rm co}^{w^*}(A\cup q B_{X^*})$, $q\in \mathbb{Q}^{+}$
where $\mathbb {Q}^{+}=\{r\in \mathbb {Q}:r>0\}$. Then $A=\bigcap\limits_{q\in \mathbb{Q}^+}A_q$.
Define
\begin{equation}
\tilde{\mathcal T}(A)=\bigcap\limits_{q\in \mathbb{Q}^+}\mathcal T(A_q).
\end{equation}
We claim that $\tilde{\mathcal T}$ is well defined and it is a fully order preserving mapping defined from ${\mathscr C}_{{\rm symm,c},*}(X^*)$ onto itself.
Indeed, choose any $D\in {\mathscr C}_{{\rm ssymm,c},*}(X^*)$, then there exist $\lambda_i\in \mathbb{Q}^+, i=1, 2$ such that
\begin{equation}\nonumber
\lambda_1B_{X^*}\subseteq D\subseteq \lambda_2B_{X^*}.
\end{equation}
Let $A^{\prime}_q=A\vee (qD)$, then
\begin{equation}\nonumber
A_{\lambda_1q}\subseteq A^{\prime}_q\subseteq A_{\lambda_2q},
\end{equation}
where $A_{\lambda_iq}=A\vee (\lambda_i q B_{X^*}), i=1, 2$. Since $\mathcal T$ is a fully order preserving mapping, then $\mathcal T A_{\lambda_1q}\subseteq \mathcal T A^{\prime}_q\subseteq \mathcal T A_{\lambda_2q}$, which entails that $\tilde {\mathcal T}(A)=\bigcap\limits_{q\in\mathbb {Q}^+}\mathcal F(A^{\prime}_q)$. This implies that the image of $\tilde {\mathcal T}$ does not dependent on the choice of $D\in{\mathscr C}_{{\rm ssymm,c},*}(X^*)$, hence $\tilde {\mathcal T}$ is
well defined.

For arbitrary $M\in {\mathscr C}_{{\rm symm,c},*}(X^*)$, $q\in\Q^{+}$, let $M_q=M\vee (qB_{X^*})$, then there exists $K^{\prime}_q\in {\mathscr C}_{{\rm ssymm,c},*}(X^*)$
such that $\mathcal T(K^{\prime}_q)=M_q$. Let $K=\bigcap\limits_{q\in \mathbb{Q}^+}K^{\prime}_q$, we now prove that
$\tilde{\mathcal T}(K)=M$.

Indeed, for arbitrary $q\in \mathbb {Q}^+$,$K^{\prime}_q\in {\mathscr{C}}_{{\rm ssymm,c},*}(X^*)$, then there exist $\lambda_{i}\in\mathbb {Q^+}, i=1, 2$ such that $\lambda_1 B_{X^*}\subseteq K^{\prime}_q\subseteq \lambda_2 B_{X^*}$, then
\begin{equation}\nonumber
K_{\lambda_1}\subseteq K^{\prime}_q\subseteq K_{\lambda_2},
\end{equation}
where $K_{\lambda_i}=K\vee (\lambda_i B_{X^*}), i=1, 2$. Then $\mathcal T(K_{\lambda_1})\subseteq M_q\subseteq \mathcal T
(K_{\lambda_2})$, which implies that  $\tilde {\mathcal T}(K)=\bigcap\limits_{\lambda\in \mathbb{Q}^+}\mathcal T(K_\lambda)=\bigcap\limits_{q\in\mathbb{Q}^+}M_q=M$. Hence $\tilde{\mathcal T}$ is surjective. Since $\mathcal T^{-1}$ is order preserving which implies the injective of $\tilde {\mathcal T}$. Therefore, $\tilde{\mathcal T}$ is a fully order preserving mapping of ${\mathscr C}_{{\rm symm,c},*}(X^*)$.
Consequently, $\overline {T}=\mathcal S\mathcal T\mathcal D$ is a fully order preserving mapping from ${\rm C}_{semn}(X)$ onto itself.
Hence, there exists a unique automorphism $U:X\rightarrow X$, so that
\begin{equation}\nonumber
(\overline Tf)(x)=f(Ex),{\rm for\; all}f\in {\rm C}_{semn}(X), x\in X;
\end{equation}
where $E\in\{\pm U\}$.
Therefore, there exists a fully order preserving mapping $S:{\rm semn}(X)\rightarrow{\rm semn}(X)$ so that
$S_{|\rm norm(X)}=T$


{\rm ii)}. Suppose that  $T:{\rm semn}(X)\rightarrow {\rm semn}(X)$ is a fully  order preserving mapping. By Theorem \ref{representation-semi}, $T$ is again fully order preserving on ${\rm norm}(X).$
\end{proof}

Recall that an extended real-valued function $f$ defined on a Banach space $X$ is said to be positively (resp. absolutely) homogeneous of degree $p$ if it satisfies
\begin{equation}\nonumber
f(\lambda x)=\lambda^ {p}f(x),~~~\forall \lambda\geq0,~ x\in X.
\end{equation}
 \begin{equation}\nonumber
({\rm resp.}\;\; f(\lambda x)=|\lambda|^ {p}f(x),~~~\forall \lambda\in\mathbb R,~ x\in X.)
\end{equation}
We denote by ${\rm ph}^{p}(X)$ (resp. ${\rm ah}^{p}(X)$), the cone of l.s.c. positively (resp. absolutely) homogeneous convex functions of degree $p$.

\begin{corollary}
Suppose that $T: {\rm ph}^{p}(X)\rightarrow {\rm ph}^{p}(X)$ (resp. ${\rm ah}^{p}(X)\rightarrow {\rm ah}^{p}(X)$) is a fully order preserving mapping. Then there is a unique isomorphism $U: X\rightarrow X$ so that
\begin{equation}\label{representation-ph}
T(f)(x)=f(Ux),\;\;{\forall} \;f\in {\rm ph}^{p}(X), \;x\in X.
\end{equation}
\begin{equation}\nonumber
({\rm resp.} \;\;T(f)(x)=f(\pm Ux),\;\;{\forall} \;f\in {\rm ah}^{p}(X), \;x\in X.)
\end{equation}

\end{corollary}
\begin{proof}
Since a l.s.c. convex function $f$ is positively (resp. absolutely) homogeneous of degree $p$ ($1\leq p<\infty$) if and only if it is of the form $f=k^{p}$ for some l.s.c. Minkowski functional (resp. seminorm) $k$ (see, for instance \cite{Ro}),
\begin{equation}\label{p-homogeneous}
\varphi(f^{1/p})(x)=(Tf)^{1/p}(x),\;f\in{\rm ph}^{p}(X)\;({\rm resp.}\;{\rm ah}^{p}(X)), \;x\in X
\end{equation}
defines a fully order preserving mapping $\varphi:{\rm mink}(X)\rightarrow {\rm mink}(X)$ (resp. ${\rm semn}(X)\rightarrow {\rm semn}(X)$). By Theorem \ref{representation-mink} (resp. Theorem \ref{representation-semi}), there is a unique isomorphism $U:X\rightarrow X$ so that
\begin{equation}\nonumber
\varphi(f^{1/p})(x)=f^{1/p}(Ux),\;\;f\in{\rm ph}^{p}(X), \;x\in X.
\end{equation}
\begin{equation}\nonumber
({\rm resp.} \;\;\varphi(f^{1/p})(x)=f^{1/p}(\pm Ux),\;\;f\in{\rm ah}^{p}(X), \;x\in X.)
\end{equation}
This combining with  \eqref{p-homogeneous} entail that \eqref{representation-ph} holds.
\end{proof}

\section{Fully order preserving mappings on the cone sublinear functions}
We conclude this paper with a representation theorem of fully order preserving mappings defined on cone consisting of sublinear functions, i.e. positively homogenous convex functions.

The following theorem is a  more general version of the representation theorem of fully order preserving mappings on the cone of sublinear functions.

\begin{theorem}\label{representation for sublinear functions}
Suppose that $C\in\mathfrak{C}(X)$ is a bounded sup-complete, which consists of continuous convex functions, and $T:C\rightarrow C$ is a fully order preserving mapping.
Assume that $C$ admits a  $w^*$-closed  sup-generating class $G\subseteq X^*$, which is a subspace of $X^*$. Then there exists a continuous linear mapping $U: X\rightarrow X$ satisfying  $\tilde{U}: X/F\rightarrow X/F$ (defined by
\[
\tilde{U}(x+F)=U(x),\;\; x+F\in X/F)
\]
is an linear isomorphism, and $\varphi_0\in X^*$ so that
\begin{equation}\label{representation-subl}
(Tf)(x)=f(Ux)+\langle\varphi_0,x\rangle,\;{\rm for\;all\;}x\in X,
\end{equation}
where $F={^\bot G}\coloneqq\{x\in X:\langle\varphi,x\rangle=0,\;{\rm for\;all}\;\varphi\in G\}.$
\end{theorem}
\begin{proof}
Since $G\in\mathfrak{G}_2\subseteq\mathfrak{G}_1$ is a $w^*$-closed subspace of $X^*$, $G$ is a perfect generating class of $C$. By Theorem \ref{sublinear}, $T$ is $w^*$-continuous affine on $C$ with $TG=G$.
Write $T(0)=\varphi_0\in G$, and let $S=T|_G-\varphi_0$. Then $S$ is a fully order preserving $w^*$-to-$w^*$ continuous linear operator on $G$.
Put $F={^\bot G}$. Then $(X/F)^*=F^\bot=G$.
Let $U:X/F\rightarrow X/F$ be defined for $x$ by
\[
\langle S\varphi,x\rangle=\langle\varphi,Ux\rangle,\;\;{\rm for\;all\;}\varphi\in G.
\]
Then it is easy to see that $U$ is a bounded linear operator with $U^*=S$. Since $S$ is bijective, it yields that $\tilde{U}:X/F\rightarrow X/F$ is bijective, hence, a isomorphism. Therefore, $S\varphi=\varphi U$ for all $\varphi\in G$. Now, given $f\in C$, by Lemma \ref{continuity} {\rm i)},
\[
Sf=\sup\limits_{h\in B_{Sf}(G)}h=f(U\cdot),
\]
that is, \eqref{representation-subl} holds.
\end{proof}
The following corollary is immediately from Theorem \ref{representation for sublinear functions} by letting $C={\rm subl}(X)$.
\begin{corollary}\label{sublinear representation theorem}
Suppose that $T: {\rm subl}(X)\rightarrow {\rm subl}(X)$ is a fully order preserving mapping. Then  there exists a linear  isomorphism $U: X\rightarrow X$, and $\varphi_0\in X^*$  so that
\[
(Tf)(x)=f(Ux)+\langle\varphi_0,x\rangle,\;{\rm for\;all\;}x\in X.
\]
\end{corollary}

\vskip 3mm
\noindent\textbf{Acknowledgements}\quad \small{This work was supported  by the Natural Science Foundation of China (Grant No. 11731010 \& 11371296). The authors would like to thank the referees for their constructive comments and helpful suggestions. The first named author is grateful to Professor Shangquan Bu, Professor Chunlan Jiang and Professor Quanhua Xu for their
very helpful conversations on this paper.}

\bibliographystyle{unsrt}


\end{document}